\documentclass[11pt,reqno]{amsart}
\usepackage{amsthm,amsmath,amssymb,latexsym,soul,cite,mathrsfs}
\usepackage{enumitem}
\usepackage{color,enumitem,graphicx}
\usepackage[colorlinks=true,urlcolor=blue,
citecolor=red,linkcolor=blue,linktocpage,pdfpagelabels,
bookmarksnumbered,bookmarksopen]{hyperref}
\usepackage[english]{babel}

\usepackage[left=2.7cm,right=2.7cm,top=2.9cm,bottom=2.9cm]{geometry}

\usepackage[hyperpageref]{backref}

\numberwithin{equation}{section}

\theoremstyle{plain}
\newtheorem{theorem}{Theorem}[section]
\newtheorem{theoremletter}{Theorem}

\newtheorem{lemma}[theorem]{Lemma}
\newtheorem{lemmaletter}[theoremletter]{Lemma}
\newtheorem{corollary}[theorem]{Corollary}
\newtheorem{proposition}[theorem]{Proposition}
\newtheorem{propositionletter}[theoremletter]{Proposition}
\theoremstyle{remark}
\newtheorem{remark}[theorem]{Remark}
\newtheorem{example}[theorem]{Example}
\theoremstyle{definition}
\newtheorem{definition}[theorem]{Definition}
\newtheorem{definitionletter}[theoremletter]{Definition}

\renewcommand{\epsilon}{\varepsilon}

\renewcommand{\theta}{{\vartheta}}

\renewcommand{\rightarrow}{\to}

\newcommand{\dx}{\,\mathrm{d}x}
\newcommand{\dy}{\,\mathrm{d}y}
\newcommand{\dxi}{\,\mathrm{d}\xi}
\newcommand{\dsigma}{\,\mathrm{d}\varsigma}
\newcommand{\dtau}{\,\mathrm{d}\tau}
\newcommand{\dxdy}{\,\mathrm{d}x\mathrm{d}y}

\newcommand{\dive}{\mathrm{div}}
\DeclareMathOperator{\supp}{supp}

\DeclareMathOperator{\sgn}{sgn}
\newcommand{\rad}{\mathrm{rad}}
\newcommand{\loca}{\operatorname{loc}}

\title[Fractional scalar field equations]{Concentration-compactness principle for nonlocal scalar field equations with critical growth}

\author[J.M.\ do \'O]{Jo\~ao Marcos do \'O}
\author[D.~Ferraz]{Diego Ferraz}

\address[J.M. do \'O]{Department of Mathematics,
Federal University of Para\'{\i}ba
\newline\indent
58051-900, Jo\~ao Pessoa-PB, Brazil}
\email{\href{mailto:jmbo@pq.cnpq.br}{jmbo@pq.cnpq.br}}

\address[D.~Ferraz]{Department of Mathematics,
	Federal University of Para\'{\i}ba
	\newline\indent
	58051-900, Jo\~ao Pessoa-PB, Brazil}
\email{\href{mailto:diego.ferraz.br@gmail.com}{diego.ferraz.br@gmail.com}}

\thanks{Research supported in part by INCTmat/MCT/Brazil, CNPq and CAPES/Brazil}
 
\begin{document}

%%%%%%%%%%%%%%%%%%%%%%%%%%%%%%%%%%%%%%%%%%%%%%%%%%%%%%%%%%%%%%%%%%%%%%%%%%%%%%%%%%%%%%%%%%%%%%%%%%%%%%%%%%%%%%%%%%%%%%%%%%%%%%%%%%%%%%%%%%%%%%%%%%%%%%%%%%%%%%%%
%
%                               ABSTRACT
%
%%%%%%%%%%%%%%%%%%%%%%%%%%%%%%%%%%%%%%%%%%%%%%%%%%%%%%%%%%%%%%%%%%%%%%%%%%%%%%%%%%%%%%%%%%%%%%%%%%%%%%%%%%%%%%%%%%%%%%%%%%%%%%%%%%%%%%%%%%%%%%%%%%%%%%%%%%%%%%%%

\begin{abstract}
The aim of this paper is to study a concentration-compactness principle for homogeneous fractional Sobolev space $\mathcal{D}^{s,2} (\mathbb{R}^N)$ for $0<s<\min\{1,N/2\}.$ As an application we establish Palais-Smale compactness for the Lagrangian associated to the fractional scalar field equation $(-\Delta)^{s} u = f(x,u)$ for $0<s<1.$ Moreover, using an analytic framework based on $\mathcal{D}^{s,2}(\mathbb{R}^N),$ we obtain the existence of ground state solutions for a wide class of nonlinearities in the critical growth range.
\end{abstract}

\maketitle

%%%%%%%%%%%%%%%%%%%%%%%%%%%%%%%%%%%%%%

%\bigskip
%\begin{center}
%\begin{minipage}{8cm}
%\footnotesize
%\tableofcontents
%\end{minipage}
%\end{center}

%%%%%%%%%%%%%%%%%%%%%%%%%%%%%%%%%%%%%%
%\bigskip

%%%%%%%%%%%%%%%%%%%%%%%%%%%%%%%%%%%%%%%%%%%%%%%%%%%%%%%%%%%%%%%%%%%%%%%%%%%%%%%%%%%%%%%%%%%%%%%%%%%%%%%%%%%%%%%%%%%%%%%%%%%%%%%%%%%%%%%%%%%%%%%%%%%%%%%%%%%%%%%%
%
%                               Introduction
%
%%%%%%%%%%%%%%%%%%%%%%%%%%%%%%%%%%%%%%%%%%%%%%%%%%%%%%%%%%%%%%%%%%%%%%%%%%%%%%%%%%%%%%%%%%%%%%%%%%%%%%%%%%%%%%%%%%%%%%%%%%%%%%%%%%%%%%%%%%%%%%%%%%%%%%%%%%%%%%%%

\section{Introduction}
The main goal of the present work is to analyze a concentration-compactness principle for homogeneous fractional Sobolev spaces. As an application, we address questions on compactness of the associated energy functional to the following nonlocal scalar field equation
\begin{equation}\label{P}
(-\Delta)^{s} u = f(x,u)\quad \text{in }\mathbb{R}^N,
\tag{$\mathcal{P}_{s}$}
\end{equation}
where $0<s<1,$ and the nonlinearity $f(x,t)$ is supposed to be a \textit{asymptotic self-similar function} (see Sect. 3.1 for the precise definition), which is a variation of the critical nonlinearity. Here $(-\Delta)^{s}$ is the fractional Laplacian defined by the relation
\begin{equation*}
	\mathscr{F} \left((-\Delta )^{s} u \right) (\xi) = \left|\xi\right| ^{2s} \mathscr{F}u (\xi),\ \xi \in \mathbb{R}^N,
\end{equation*}
where  $\mathscr{F}u$ is the Fourier transform of $u,$ i.e.
\begin{equation}\label{fourierdef}
	\mathscr{F} u (x)= \frac{1}{(2\pi)^{N/2}} \int _{\mathbb{R}^N} u (\xi) e ^{- i \xi \cdot x} \dxi,\ x \in \mathbb{R}^N.
\end{equation}
Let $\mathscr{S}$ be the Schwartz space consisting of rapidly decaying $C^\infty$ functions in $\mathbb{R}^N$ which, together with all their derivatives, vanish at the infinity faster than any power of $|x|.$ Equivalently, if $u\in \mathscr{S}$ the fractional Laplacian of $u$ can be computed by the following singular integral
\begin{equation*}\label{lap_frac}
(-\Delta )^{s} u (x) = C(N,s) \lim _{\varepsilon \rightarrow 0 ^{+} } \int _{ \mathbb{R}^N \setminus B_{\varepsilon } (0) } \frac{u(x) - u(y)}{|x-y|^{N + 2s}}\dy,
\end{equation*}
for a suitable positive normalizing constant
\begin{equation*}
C(N,s) = \left(\int _{\mathbb{R}^N} \frac{1 - \cos \varsigma _1 }{|\varsigma| ^{N+2s}} \dsigma\right)^{-1}.
\end{equation*}
We refer to \cite{landkof,silvestre,hitchhiker} for an introduction to the fractional Laplacian operator.

During the past years there has been a considerable amount of research involving nonlocal nonlinear stationary Schr\"{o}dinger problems. This equation arises in the study of the fractional Schr\"{o}dinger equation when looking for standing waves. Indeed, when $u$ is a solution of Eq.~\eqref{P}, it can be seen as stationary states (corresponding to solitary waves) in nonlinear equations of Schr\"{o}dinger type
\begin{equation*}
i \phi _t - (- \Delta) ^s \phi  +  f(x,\phi)=0\quad\text{in } \mathbb{R}^N.
\end{equation*}
Fractional Schr\"{o}dinger equations are also of interest in quantum mechanics (see e.g. the appendix in \cite{delpino} for details and physical motivations). Moreover, we refer to \cite{levy2}, \cite{levy1} and \cite{levy3}, where equations involving the operator $(- \Delta) ^s$ arise from several areas of science such as biology, chemistry or finance.

A lot of work has been devoted to the existence of solutions for nonlinear scalar field equations like Eq.~\eqref{P}, both for local case $(s=1)$  and nonlocal case $0<s<1,$ since the celebrated works of H.~Berestycki and P.-L. Lions \cite{berlions,berlions-II}. 
In these two papers, the authors discuss the existence of radial solutions of the semi-linear elliptic equation 
\begin{equation}\label{BL-eq}
-\Delta u = g(u), \quad u\in H^1(\mathbb{R}^N)(N \geq 3),
\end{equation}
where $g:\mathbb{R}\rightarrow \mathbb{R}$ is a continuous odd function with subcritical growth. Under some appropriate conditions on $g(t)$, they used minimizing arguments to prove (in part I) the existence of a positive radial ground state for \eqref{BL-eq}, that is, a solution having the property of the least action among all possible solutions. In \cite{tintapaper}, K. Tintarev  has treated the non-autonomous problem $-\Delta u = g(x,u)$ in $\mathbb{R}^N(N \geq 3)$, $u\in \mathcal{D}^{1,2}(\mathbb{R}^N),$ when the nonlinearity $g(x,t)$ is allowed to have critical growth with asymptotically self-similar oscillations about the critical power $|t|^{2^*-2}t$. Recently, using some minimax arguments, X.~Chang and Z-Q.~Wang \cite{chang-wang} proved  the existence of a positive ground state for fractional scalar field equations of the form $(-\Delta )^{s}u = g(u)$ in $\mathbb{R}^N (N\geq 2), \, s \in (0,1),$ when $g(t)$ has subcritical growth and satisfies the Berestycki--Lions type assumptions. In \cite{DOO-SQUASSINA-ZHANG}, J.~M.~do~\'{O} et al., established the existence of ground state solutions to the fractional scalar field equation $(-\Delta )^{s}u = g(u)$ in $\mathbb{R}^3, \, s \in (0,1),$ when $g(t)$ has critical growth.

Motivated by the results cited above, another important purpose of this work is to prove the existence of ground state solutions for the nonlinear scalar field equation \eqref{P} in the ``zero mass case'' with nonlinearities in the critical growth range. It is well known that Eq.~\eqref{P} admits a variational setting in fractional Sobolev spaces, and the solutions are constructed with a variational method by a minimax procedure on the associated energy functional. However, we note that the usual variational techniques cannot be applied straightly because of a lack of compactness, which roughly speaking, originates from the invariance of $\mathbb{R}^N$ with respect to translation and dilation and, analytically, appears because of the non-compactness of the Sobolev embedding. For instance, it is not possible to apply the minimax type arguments used by P.~Felmer et al. \cite{fqt} and R.~Servadei and E.~Valdinoci \cite{bounded,servadei_MP,brezis-niremberg} (where some new techniques of nonlinear analysis have been introduced to solve nonlocal equations) because their approach rely strongly on the sub-criticality of the nonlinear terms or the  boundedness of the domain.
To overcome these difficulties, under appropriate assumptions, we establish a profile decomposition for bounded sequences of suitable Sobolev spaces, which is based in the profile given in \cite{palatucci}, proving that Palais-Smale sequences of the associated energy functional converges, up to a subsequence, in a similar way to the global compactness studied in \cite{global_pala,global_brasco}.

The idea for proving the existence of ground state solution for Eq.~\eqref{P} in the autonomous case is based in a constrained minimization argument similar to \cite{berlions}.  We obtain the result by using the invariance of the problem with respect to action of the translation and dilation group in $\mathbb{R}^N,$ thanks to our concentration-compactness principle and a specific Pohozaev identity.
%We use a general argument in order to obtain a specific Pohozaev identity 
Our argument allow us to avoid the typical assumption that $t^{-1}f(x,t)$ is an increasing function, which is usually required in the approach of constrained minimization over a Nehari manifold. Moreover, to prove the existence for the autonomous case $f(x, t) = f(t)$, we do not require the well known Ambrosetti-Rabinowitz condition.

%The same Pohozaev identity allows us to study the autonomous case $f(x,t) = f(t)$ of Eq.~\eqref{P}, from which we may derive our existence results for the general case. The proof of this identity is essentially based in the use of the so called $s$-harmonic extension introduced by L. Caffarelli and L. Silvestre \cite{caf_silv} and remarks contained in \cite{moustapha} and \cite{frac_niremberg}. 

The existence for the general case of Eq.~\eqref{P} is obtained from the autonomous case which was derived from a class of Pohozaev identity for the space $\mathcal{D}^{s,2}(\mathbb{R}^N)$. The proof of this identity is essentially based in the use of the so called $s$-harmonic extension introduced by L. Caffarelli and L. Silvestre \cite{caf_silv} (see also \cite{molvc}, where a similar procedure was introduced by using probabilistic methods) and remarks contained in \cite{moustapha} and \cite{frac_niremberg}.

Our main results may be seen as the nonlocal counterpart of some theorems of K.~Tintarev et al. \cite{tinta_pos,tintapaper,tintabook}. In comparison with the local case \cite{tintapaper}, we also mention some additional difficulties: the Pohozaev type identities for the fractional framework available in the literature (cf. \cite{chang-wang,frank,rosoton}) do not match with our settings; an additional hypothesis (assumption \eqref{f_extra}) must be considered in order to achieve the concentration-compactness for the non-autonomous case. In fact, the asymptotic additivity \eqref{f_extra} takes the role to describe precisely the behavior of weak convergence under our settings (Proposition \ref{lemma2}). At this point a natural question arises: Is hypothesis \eqref{f_extra} necessary to describe the limit of the profile decomposition terms (see Theorem~\ref{teo_tinta_frac})? Indeed, we believe that without condition \eqref{f_extra} it is possible to find examples for which this description fails.

To the best of our knowledge, this is the first work that shows a Pohozaev type identity for the homogeneous Sobolev space $\mathcal{D}^{s,2}(\mathbb{R}^N)$ and for $f(t)$ in the critical growth range. Our method is very convenient in the sense that with our arguments we can always derive a Pohozaev type identity in the fractional framework without relying in global regularization of the solutions. In the present literature, there are only Pohozev type identities for solutions in the inhomogeneous fractional Sobolev space $H^s(\mathbb{R}^N),$ $0<s<\min\{1,N/2\},$ and for $f(t)$ with subcritical growth (cf. \cite{chang-wang}).  Moreover, the argument for the proof relies in obtaining the behavior of solutions in the whole space $\mathbb{R}^N$  (cf. \cite{frank}).

In addition, we introduce a new class of nonlinearities of the critical growth type for the fractional framework, that does not satisfies that $t^{-1}f(x,t)$ is an increasing function and include the power $|t|^{2_s ^\ast-2}t$ as an example. We believe that this new notion of criticality together with our concentration-compactness, can lead to a new way to approach elliptic problems involving nonlinearities with critical growth and the fractional Laplacian, for instance, replacing the well known nonlinearity $f(x,t)=K(x)|t|^{2_s ^\ast -2}t$ for a general \textit{self-similar} function under our settings (see Sect. \ref{s_Hypothesis}) when dealing with general equations that involves critical growth.

Moreover, as it is well known, one of the main difficulties in leading with nonlinearities with critical growth condition is proving that the minimax level of the functional associated to Eq.~\eqref{P} avoids levels of non-compactness, which usually requires additional description of the nonlinearity growth. We avoid this by considering that $f(x,t)$ has appropriated limits consistent with our concentration-compactness and comparing the minimax level of functional associated to Eq.~\eqref{P} with the limit ones.
%%%%%%%%%%%%%%%%%%%%%%%%%%%%%%%%%%%%%%%%%%%%%%%%%%%%
%	Outline
%%%%%%%%%%%%%%%%%%%%%%%%%%%%%%%%%%%%%%%%%%%%%%%%%%%%
\subsection{Outline} The paper is organized as follows. Sect.~\ref{Weak Convergence} is devoted to the description of the profile decomposition of bounded sequence in the Sobolev space $\mathcal{D}^{s,2}(\mathbb{R}^N)$. In Sect.~\ref{Application}, we give some applications of Theorem~\ref{teo_tinta_frac} to study the existence of mountain-pass solutions of Eq.~\eqref{P}, for the autonomous and non-autonomous case. In Sect.~\ref{Preliminaries}, we state some basic results (without prove) on the fractional Sobolev space $\mathcal{D}^{s,2}(\mathbb{R}^N)$. We also prove a Pohozaev identity for weak solutions of Eq.~\eqref{P} and establish the background material to develop the profile decomposition described in Sect. \ref{Weak Convergence}. In Sect.~\ref{assumptions}, we prove our concentration-compactness result. In Sect.~\ref{self-similar_functions}, we study the basic properties for a class of nonlinearities in the critical growth range dealt in this paper, which is fairly used to establish the results regarding the Eq. \eqref{P}. In Sect.~\ref{onbehaviour}, using the properties obtained in the Sect. \ref{self-similar_functions}, we describe the limit of the profile decomposition of the Palais-Smale sequence at the mountain pass level of the Lagrangian of Eq.~\eqref{P}. In Sect.~\ref{asfec}, we prove the results given in Subsect. \ref{stat_auto} and describe some properties regarding the minimax levels associated with the functional energy of Eq.~\eqref{P}, for the autonomous case. In Sect.~\ref{tnac}, we prove our result about the existence of non-trivial weak solution of Eq.~\eqref{P} in the non-autonomous case, and for the sake of discussion, we establish a sufficient condition that ensures one of our hypothesis, precisely, the assumption \eqref{suficient}.

%%%%%%%%%%%%%%%%%%%%%%%%%%%%%%%%%%%%%%%%%%%%%%%%%%%%
%		Profile decomposition
%%%%%%%%%%%%%%%%%%%%%%%%%%%%%%%%%%%%%%%%%%%%%%%%%%%%
\section{Profile decomposition in $\mathcal{D}^{s,2}(\mathbb{R}^N)$}\label{Weak Convergence}
For $0<s<\min\{1,N/2\},$ let us consider the homogeneous fractional Sobolev space $\mathcal{D}^{s,2}(\mathbb{R}^N),$ which is defined as the completion of $C ^{\infty } _0 (\mathbb{R}^N)$ under the norm 
\begin{equation*}
\|u\|^2 :=  \int _{\mathbb{R}^N} |\xi|^{2s} \left |\mathscr{F} u \right| ^2 \dxi.
\end{equation*}
It is well know that $\mathcal{D}^{s,2}(\mathbb{R}^N)$ is continuous embedded in $L^{2^\ast_s}(\mathbb{R}^N),$ $0<s<\min\{1,N/2\},$ where $2_s^\ast = 2N/(N-2s).$ The aforementioned concentration-compactness is made by means of profile decomposition for bounded sequences in homogeneous fractional Sobolev spaces, which can be considered as extensions of the Banach-Alaoglu theorem. 
This kind of profile decomposition has been widely investigated in various settings, for instance we may cite the ones in
\cite{solimini,gerard,jaffard,struwe,palatucci}. It describes how the convergence of a bounded sequence fails in the considered space. In \cite{gerard} P. G\'{e}rard introduced the subject in the fractional framework, and in \cite{palatucci} G. Palatucci and A. Pisante studied this matter based in the abstract version of profile decomposition in Hilbert spaces due to K.~Schindler and K.~Tintarev \cite{tintabook}. Moreover, in \cite{palatucci} the problem of the cocompactness embedding of $\mathcal{D}^{s,2}(\mathbb{R}^N)$ in $L^{2^\ast_s} (\mathbb{R}^N)$ with respect to the group of dilations and translations in the sense of \cite{cocompact} (see \cite[Proposition 1]{palatucci}) was extensively discussed. Our result stated here is based in the abstract approach made in \cite{palatucci,tintabook} with some adjustments.
\begin{theorem}\label{teo_tinta_frac}
	Let $(u_k) \subset \mathcal{D}^{s,2}(\mathbb{R}^N)$ be a bounded sequence, $\gamma > 1$ and $0<s<\min\{1,N/2\}.$ Then there exists $\mathbb{N}_{\ast } \subset \mathbb{N},$ disjoints sets (if non-empty) $\mathbb{N} _{0}, \mathbb{N} _{- }, \mathbb{N}_ {+} \subset \mathbb{N}, $ with $\mathbb{N}_{\ast } = \mathbb{N} _{0} \cup  \mathbb{N}_ {+} \cup \mathbb{N} _{-} $ and sequences $(w ^{(n)}) _{n \in \mathbb{N}_{\ast } } \subset \mathcal{D}^{s,2}(\mathbb{R}^N),$ $(y_k ^{(n)}) _ {k \in \mathbb{N}} \subset \mathbb{Z}^N,$ $(j _k ^{(n)}) _{k \in \mathbb{N}}  \subset \mathbb{Z},$ $n \in \mathbb{N}_{\ast },$ such that, up to a subsequence of $(u_k),$
	\begin{align}
	\qquad \qquad &\gamma ^{-\frac{N-2s}{2}j_k ^{(n)}}u_k \big(\gamma ^{- j_k ^{(n)}} \cdot + y_k^{(n)} \big)\rightharpoonup w^{(n)} \text{ as } k \rightarrow \infty,\mbox{ in } \mathcal{D}^{s,2}(\mathbb{R}^N),&\label{seis.um}\\ 
	\qquad \qquad &\big|j_k ^{(n)} - j_k ^{(m)}\big|+\big|\gamma ^ {j_k ^{(n)}}( y^{(n)}_k - y^{(m)}_k )\big| \rightarrow \infty , \text{ as } k \rightarrow \infty,\text{ for } m \neq n,&\label{seis.dois} \\ 
	\qquad \qquad &\sum _{n \in \mathbb{N}_{\ast}}\big \| w^{(n)} \big \|^2  \leq \limsup_k \|u_k \| ^2,&\label{seis.tres} \\ 
	\qquad \qquad & u_k - \sum _{n \in \mathbb{N}_{\ast }} \gamma ^{\frac{N-2s}{2} j_k ^{(n)} } w^{(n)}\big(\gamma^{j _k ^{(n)}} ( \cdot - y_k^{(n)} ) \big) \rightarrow  0, \text{ as } k \rightarrow \infty, \text{ in } L^{2^{\ast } _s}(\mathbb{R}^{N}),&\label{seis.quatro}
	\end{align}
	and the series in \eqref{seis.quatro} converges uniformly in $k.$ Furthermore, $1 \in \mathbb{N} _0,$ $y_k ^{(1)} = 0;$ $j _k ^{(n)} = 0 $ whenever $n \in \mathbb{N}_0;$ $j_k ^{(n)} \rightarrow -\infty$ whenever $n \in \mathbb{N}_{- };$ and $j_k ^{(n)} \rightarrow +\infty$ whenever $n \in \mathbb{N}_{ +}.$	
\end{theorem}

This decomposition is unique up to permutation of indices, and up to constant operator multiples (See \cite[Proposition 3.4]{tintabook}).
As it can be seen, Theorem \ref{teo_tinta_frac} describes how the convergence of bounded sequences in $\mathcal{D}^{s,2}(\mathbb{R}^N)$ fail to converge in $L^{2_s^\ast}(\mathbb{R}^N)$.  This ``error'' of convergence is generated, roughly speaking, by the invariance of action of the group of translation and dilation in $\mathcal{D}^{s,2}(\mathbb{R}^N).$ Observe that the behavior for the correction term in \eqref{seis.quatro} is precisely described in the assertions \eqref{seis.um}--\eqref{seis.tres}.

Here we point out differences from our Theorem \ref{teo_tinta_frac} and some results on profile decompositions contained in \cite{gerard,palatucci}. The decomposition in Theorem \ref{teo_tinta_frac} is based in a discrete group of operators, that is, the dilations in the following form 
\begin{equation}\label{our_dilation}
\delta_j u(x) =\gamma ^{\frac{N-2s}{2} j} u (\gamma ^j x ),\ \gamma >1,\ j \in \mathbb{Z}.
\end{equation}
From \eqref{our_dilation}, we can decompose (in a similar way as in \cite[Theorem 5.1]{tintabook}) the collection of the ``dislocated profiles'' $w^{(n)}$ in three: dilation by ``enlargement'' ($\mathbb{N}_{-}$), dilation by ``reducement'' ($\mathbb{N}_{+}$), and no dilation (pure translation $\mathbb{N}_{0}$). This allow us to study scalar field equations involving nonlinearities with critical growth more general than the pure critical power (see Sect. \ref{self-similar_functions}), the so called {\it asymptotic self-similar functions}  (assumptions \eqref{f_extra}--\eqref{f_pohozaev}, \eqref{selfsimilar}). On the other hand in \cite{gerard,palatucci} was considered continuous dilations of the form
\begin{equation*}
\delta _{\lambda} u(x) = \lambda ^{\frac{N-2s}{2} } u(\lambda x),\ \lambda >0,
\end{equation*}
and their decomposition holds for all $0<s<N/2.$ We should mention that Theorem \ref{teo_tinta_frac} holds for $0<s \leq 1$ and at this point arise a natural question which is to prove this result for $1<s<N/2.$ In \cite[Proposition 1]{palatucci} it was proved that $D_{\mathbb{R}^N}$--weak convergence is equivalent to strong convergence in $L^{2_s ^\ast} (\mathbb{R}^N),$ for $0<s<N/2$ (see Section 4.3 below), where for $\gamma > 1$ given,
\begin{equation}\label{D_erren}
D_{\mathbb{R} ^N} := \left\lbrace d_{y,j} : \mathcal{D}^{s,2} (\mathbb{R}^N) \rightarrow \mathcal{D}^{s,2} (\mathbb{R}^N) : d_{y,j} u (x) := \gamma ^{\frac{N-2s}{2} j} u (\gamma ^j (x-y)),\ y \in \mathbb{R}^N, \ j \in \mathbb{R} \right\rbrace.
\end{equation}
From this we can conclude that the answer to that question is analogously to prove that $D_{\mathbb{R}^N}$--weak convergence is equivalent to the $D_{\mathbb{Z}^N}$--weak convergence in $\mathcal{D}^{s,2}(\mathbb{R}^N),$ for $1<s<N/2,$ where $D_{\mathbb{Z} ^N} := \left\lbrace d_{y,j} : y \in \mathbb{Z}^N,\ j \in \mathbb{Z} \right\rbrace.$ In the affirmative case, Theorem \ref{teo_tinta_frac} can be seen as a corollary of the decomposition given in Theorem \cite[Theorem 4, Theorem 8]{palatucci}, with minor changes (provided also in Section \ref{assumptions}). Nevertheless, for the case that $0<s<1,$ we present a proof of this fact (given in Proposition \ref{cocompact}), which can also be seen as an alternative proof of Theorem \ref{teo_tinta_frac}.

In this work we also develop techniques to obtain non-trivial weak solutions of Eq.~\eqref{P} extending the results in \cite{tintapaper} for the nonlocal case. Also, it seems for us that Theorem \ref{teo_tinta_frac} is more appropriated to study the existence of non-trivial solutions for scalar field equation  \eqref{P} than \cite[Theorem 4, Theorem 8]{palatucci} and \cite[Theorem 1.1]{gerard}. It is not clear how one can apply \cite[Theorem~4]{palatucci} to obtain such a result for nonlinearities with asymptotically self-similar oscillations about the fractional critical growth (see Sect.~\ref{s_Hypothesis} for precise definitions).

%the so called asymptotic self-similar functions (see assumptions \eqref{f_extra}--\eqref{f_pohozaev} and Sect. \ref{self-similar_functions})
%More precisely, since in the study of existence of solutions for \eqref{P} we consider the nonlinearity $f(x,t)$ as being asymptotically  self-similar (assumptions \eqref{f_extra}--\eqref{f_pohozaev} below) is crucial in our argument that (see Sect. \ref{self-similar_functions})

%%%%%%%%%%%%%%%%%%%%%%%%%%%%%%%%%%%%%%%%%%%%%%%%%%%%
% 	Nonlinear scalar field equations
%%%%%%%%%%%%%%%%%%%%%%%%%%%%%%%%%%%%%%%%%%%%%%%%%%%%

\section{Nonlinear scalar field equations}\label{Application}

In this section we state our main results regarding the existence of solutions of Eq.~\eqref{P}. In what follows, we always assume that $0<s<\min\{1,N/2\}.$

%for the autonomous case as well as the non-autonomous case. We always assume that $0<s<\min\{1,N/2\}.$

%%%%%%%%%%%%%%%%%%%%%%%%%%%%%%%%%%%%%%%%%%%%%%%%%%%%
% 	Hypothesis
%%%%%%%%%%%%%%%%%%%%%%%%%%%%%%%%%%%%%%%%%%%%%%%%%%%%

\subsection{Hypothesis}\label{s_Hypothesis}
In order to describe our results on the energy functional of \eqref{P} in a more precisely way, we will make the following assumptions:
\begin{flalign}\label{f_1}\tag{$f_1$}
   f: \mathbb{R}^N \times \mathbb{R} \rightarrow \mathbb{R}
  \text{ satisfies  the Carath\'{e}odory conditions.} &&
\end{flalign}
\begin{flalign}\label{f_2}\tag{$f_2$}
 |f(x,t)| \leq C |t| ^ {2 ^\ast _s - 1}, \;  \forall \,  t \in \mathbb{R} \text{ and a.e. } x \in \mathbb{R}^N.  &&
\end{flalign}
\begin{flalign}\label{AR} \tag{$f_3$}
  \exists \, \mu> 2; \;   \mu F(x,t) := \mu \int _0 ^t f(x, \tau)\dtau \leq f(x,t)t, \ \forall \; t \in \mathbb{R} \text{ and a.e. } x \in \mathbb{R}^N. &&
\end{flalign}
\begin{flalign}\label{posi_algum} \tag{$f_4$}
\begin{aligned}
& \exists \, (x_0,t_0) \in \mathbb{R}^N \times \mathbb{R}_+ \text{  such that} \\
& |B_R|\inf _{B_R(x_0)} F(x,t_0) + |B_{R+1}\setminus B_R| \inf_{(x,t) \in  (B_{R+1}(x_0)\setminus B_R(x_0) ) \times [0,t_0]} F(x,t) >0
\end{aligned} &&
\end{flalign}
 In the study of the autonomous case we consider a weak version of \eqref{posi_algum} which we state next.
\begin{flalign}\label{posi_algum_auto}
\exists \, t_0\in \mathbb{R}\text{ such that }F(t_0) >0.\tag{$f'_4$}&&
\end{flalign}
\begin{flalign}\label{f_extra}\tag{$f_5$}
\begin{aligned}
& \text{ For any real numbers } a_1, \ldots, a_{M}, \exists \,  C=C(M)>0 \text{ such that } \\
& \left|F\left(x,\sum _{n=1} ^{M} a_n \right) - \sum_{n=1}^{M}F(x,a_n )\right| \leq C(M) \sum _{m \neq n \in \{1,\ldots,M\}} |a_n| ^{2^{\ast} _s -1} |a_m|\ \text{ a.e. } x \in \mathbb{R}^N.
\end{aligned} &&
\end{flalign}
\begin{flalign}\label{f_5} \tag{$f_6$}
\begin{aligned}
& \exists \, \gamma>1, \;  0<s<\min\{1,N/2\} \text{ such that the following limits exists} \\
& f_0 (t) := \lim _{|x| \rightarrow \infty} f(x,t), \\
& f_+(t):= \lim _{j \in \mathbb{Z}, j \rightarrow + \infty} \gamma ^{- \frac{N+2s}{2}j} f\left(\gamma^{-j} x , \gamma ^{ \frac{N-2s}{2}j}  t \right),\\
& f_{-}(t):= \lim _{j \in \mathbb{Z}, j \rightarrow - \infty} \gamma ^{- \frac{N+2s}{2}j} f\left(\gamma ^{-j} x, \gamma ^{ \frac{N-2s}{2}j}  t \right).\\
&\text{uniformly in }x\text{ and in compact sets for }t.
\end{aligned} &&
\end{flalign}
\begin{flalign}\label{f_pohozaev} \tag{$f_7$}
 f_0,\; f_{+}, \; f_{-}  \text{ are continuously differentiable.} &&
\end{flalign}

\bigskip

We consider associated with Eq.~\eqref{P}, the functional $I :\mathcal{D}^{s,2}(\mathbb{R}^N) \rightarrow \mathbb{R}$ given by
\begin{equation}\label{energy_functional}
I(u) = \frac{1}{2}\int _{\mathbb{R}^N} |(-\Delta)^{s/2} u| ^2 \dx -  \int _{\mathbb{R}^N} F(x,u) \dx.
\end{equation}
Assuming that $f(x,t)$ satisfies \eqref{f_2} and using the same arguments of \cite{minimax}, $I \in C^1 (\mathcal{D}^{s,2}(\mathbb{R}^N))$ and
\begin{equation*}
I'(u)\cdot v= \int _{\mathbb{R}^N} (-\Delta )^{s/2} u\  (-\Delta )^{s/2 } v \dx - \int _{\mathbb{R}^N} f(x,u) v \dx,\quad u, v \in \mathcal{D}^{s,2}(\mathbb{R}^N).
\end{equation*}
Thus critical points of $I$ correspond to weak solutions of Eq.~\eqref{P} and conversely. 

Regarding the minimax level, we consider
\begin{equation*}
\Gamma _I= \left\lbrace \gamma \in C([0, \infty ),  \mathcal{D}^{s,2} (\mathbb{R}^N) ) : \gamma(0)=0, \ \lim _{t \rightarrow \infty}I(\gamma (t)) = - \infty \right\rbrace,
\end{equation*}
and
\begin{equation}\label{minimax}
c(I)= \inf_{\gamma \in \Gamma_I} \sup_{t \geq 0 } I(\gamma (t)).
\end{equation}
For the nonlinearities $f_0, f_+, f_{-}$, we consider the associated energy functionals given by
\begin{equation*}
I_\kappa(u) = \frac{1}{2}\int _{\mathbb{R}^N} \big|(-\Delta)^{s/2} u\big| ^2 \dx -  \int _{\mathbb{R}^N} F_\kappa(u) \dx, \quad F _\kappa (t):= \int _0 ^t f _\kappa(\tau)\dtau
\end{equation*}
and the respectively minimax levels
\[
c_\kappa = \inf_{\gamma \in \Gamma_{\kappa}} \sup_{t \geq 0 } I_{\kappa}(\gamma (t))
\]
where 
\begin{equation*}
\Gamma_{\kappa} = \left\lbrace \gamma \in C([0, \infty ),  \mathcal{D}^{s,2} (\mathbb{R}^N) ) : \gamma(0)=0, \ \lim _{t \rightarrow \infty}I_{\kappa}(\gamma (t)) = - \infty \right\rbrace,
\end{equation*}
for $ \kappa=0,+,-.$ Next, we assume a condition that compares the mountain pass levels defined above, precisely, for each $\kappa=0,+,-,$ there is a real number $t_\kappa $ such that 
\begin{flalign*}\label{suficient}
F_\kappa (t_\kappa ) >0 \quad \mbox{and}\quad c(I) < c(I _\kappa).\tag{$f_8$}&&
\end{flalign*}
We also consider the additional assumption for $\kappa=0,+,-$,
\begin{flalign}
&F_\kappa\text{ satisfies \eqref{posi_algum_auto}},\ F _\kappa (t) \leq  F(x,t),\ \text{a.e. } x\in\mathbb{R}^N,\; \forall \, t \in \mathbb{R}.\label{desigualdade_geral}\\
&\text{Moreover }\exists \delta >0\text{ such that } F _\kappa (t) <  F(x,t),\ \text{ a.e. } x \in\mathbb{R}^N,\; \forall \, t \in (-\delta , \delta) \label{f'_6}\tag{$f'_8$}&&
\end{flalign}
We are going to prove in Proposition \ref{minimaxprop} that \eqref{f'_6} implies \eqref{suficient}.
To obtain our main result, we first study the autonomous case. For that we assume:
\begin{flalign*}\label{selfsimilar}
\exists \, \gamma>1, \;  0<s<\min\{1,N/2\},\text{ such that } G(t)=\gamma ^{-Nj} G\left(\gamma ^{\frac{N-2s}{2}j} t \right), \forall j \in \mathbb{Z},\ \forall t \in \mathbb{R}.\tag{$f_9$} &&
\end{flalign*}
This allows us to derive some basic results concerning the above profile decomposition, which gives the behavior upon the functional $I$ as we pass the limit over the Palais-Smale sequences. We refer to the class of functions that satisfies \eqref{selfsimilar} as \textit{self-similar} (see Section \ref{self-similar_functions}).

In this paper, we often use the notation $\Phi (u) = \int_{\mathbb{R} ^N} F(x,u) \dx$ and $\Phi _\kappa(u) = \int_{\mathbb{R} ^N} F _\kappa (u) \dx$ for $ \kappa=0,+,-,$  also, we usually refer to the assumptions \eqref{f_1}--\eqref{suficient} for the autonomous case $f(x,t)=f(t).$
\subsection{Statement of main results}\label{stat_auto}
Next, we state the main results about the autonomous case $f(x,t)=f(t).$
\begin{theorem}\label{primeira}
Suppose that $f(t)$ satisfies \eqref{f_1}, \eqref{posi_algum_auto} and \eqref{selfsimilar}, and consider
\begin{equation}\label{max}
\mathcal{S}_l = \sup _{\|u\|^2 =l} \int _{\mathbb{R}^N}F (u ) \dx.
\end{equation}
Then, for any maximizing sequence $(u_k)$ of \eqref{max} there exists $(j_k) \subset \mathbb{Z}$ and $(y_k) \subset \mathbb{Z}^N$ such that $( \gamma ^{-\frac{N-2s}{2}j_k } u_k (\gamma ^{- j_k } \cdot + y_k ) )$ contains a convergent subsequence in $\mathcal{D}^{s,2} (\mathbb{R}^N)$. In particular, the supremum in \eqref{max} is attained. Moreover, the same conclusion holds for
\begin{equation*}
\mathcal{S}_{l , +} = \sup _{\|u\|^2 =l} \int _{\mathbb{R}^N}F _{+ } (u ) \dx \quad \mbox{and}\quad 
\mathcal{S}_{l , -} = \sup _{\|u\|^2 =l} \int _{\mathbb{R}^N}F _{- } (u ) \dx,
\end{equation*}
provided that $f(t)$ satisfies \eqref{f_1}, \eqref{f_2} and \eqref{f_5}, with $F_{+}, F_{-} $ fulfilling \eqref{posi_algum_auto} and $\mathcal{S} _1> \max \{\mathcal{S}_{1,+},\mathcal{S}_{1,-}\}.$
\end{theorem}
Our next result proves that maximizers of \eqref{max} are indeed non-trivial solutions of Eq.~\eqref{P}. Moreover, the
mountain pass level \eqref{minimax} is attained. The main tool to achieve these facts is a Pohozaev type identity proved in Section \ref{pohozaevs}, which holds under the condition $0<s<1$ and taking into account the smoothness of the nonlinearity.
\begin{theorem}\label{segunda}
Assume that $f(t)\in C^1 ( \mathbb{R} )$ satisfies \eqref{f_1}, \eqref{f_2} and \eqref{posi_algum_auto}.
\begin{enumerate}[label=(\roman*)]
\item\label{item1a} If $v$ is a nonzero critical point of $I,$ then $c(I) \leq I(v);$
\item \label{item2a} If $w$ is a maximizer of $\mathcal{S} _{l_0}$ for 
 $l_0:=((2^{\ast } _s/2) \mathcal{S} _1) ^{-\frac{N-2s}{2s}},$
then $w$ is a critical point of $I.$ Moreover $$0<\max _{t \geq 0} I(w(\cdot /t))=I(w)=c(I).$$
\end{enumerate}
\end{theorem}

From Theorem~\ref{segunda}, we conclude that to obtain weak solutions for the autonomous case, only assumptions~\eqref{f_1}, \eqref{posi_algum_auto} and \eqref{selfsimilar} are needed. Moreover, we are able to prove that the minimax level is attained without the Ambrosetti-Rabinowitz condition \eqref{AR}. 

Another way to approach Eq.~\eqref{P} is by the means of constrained minimization. In fact, due to Theorem~\ref{teo_tinta_frac} we can argue as \cite{tintabook}, and thanks to the Pohozaev identity, reasoning as in \cite{berlions}, we can derive existence of a ground state solution (or least energy) for Eq.~$\eqref{P},$ that is, a solution $u$ of \eqref{P} such that $I(u) \leq I(v),$ for any other solution $v.$
\begin{theorem}\label{minimizacao}
Suppose that $f(t) \in C^1 (\mathbb{R}^N)$ satisfies \eqref{f_1}, \eqref{posi_algum_auto} and \eqref{selfsimilar}. Let $$\mathcal{G}=\left\{u \in \mathcal{D}^{s,2} (\mathbb{R}^N ):\Phi(u)=1 \right\}$$ and consider
\begin{equation}\label{min}
\mathcal{I} = \inf _{ u \in \mathcal{G}} \int _{\mathbb{R}^N} \big|(-\Delta)^{s/2} u (x)\big| ^2 \dx.
\end{equation}
Then, for any minimizing sequence $(u_k)$ of \eqref{min} there exists $(j_k) \subset \mathbb{Z}$ and $(y_k) \subset \mathbb{Z}^N$ such that $( \gamma ^{-\frac{N-2s}{2}j_k } u_k (\gamma ^{- j_k } \cdot + y_k ) )$ contains a convergent subsequence in $\mathcal{D}^{s,2} (\mathbb{R}^N)$. In particular, there exists a minimizer $w$ for \eqref{min}. Furthermore, $u(x)=w(x / \beta )$ is a ground state solution for Eq.~\eqref{P} for some $\beta>0.$ 
\end{theorem}
In the following result we prove that Palais-Smale condition at the mountain pass level holds for the general non-autonomous case.
\begin{theorem}\label{principal}
	If $f(x,t)$ satisfies \eqref{f_1}--\eqref{f_pohozaev} and \eqref{desigualdade_geral}, then Eq. \eqref{P} has a nontrival weak solution $u$ in $\mathcal{D}^{s,2}(\mathbb{R}^N)$ at the mountain pass level, that is, $I(u) = c(I).$ Moreover, if we assume additionally that \eqref{suficient} holds true intead of \eqref{desigualdade_geral}, then any sequence $(u_k)$ in $\mathcal{D}^{s,2} (\mathbb{R}^N)$ such that $I (u_k) \rightarrow c(I)$ and $I' (u_k) \rightarrow 0$ has a convergent subsequence.
\end{theorem}
\subsubsection{Remark on the hypothesis}
\begin{remark} Next we give several helpful comments concerning our assumptions.
\begin{enumerate}[label=(\roman*)] \item On assumption  \eqref{f_1},   we recall that $
   f: \mathbb{R}^N \times \mathbb{R} \rightarrow \mathbb{R} $
satisfies  the Carath\'{e}odory conditions, if for each fixed 
 $ t \in \mathbb{R}, \; \;  f(\cdot, t) $  is measurable, and  
 for a.e.  $ x \in \mathbb{R}^N, \; \; f(x,\cdot) $   is continous in $ \mathbb{R}.$
\item Condition \eqref{f_2} includes in particular nonlinearites with critical growth.
\item Assumption \eqref{AR} is the well-known Ambrosetti-Rabinowitz condition (see \cite{ambrorab,Rabi1992}).
\item In order to prove that the functional associated with Eq.~\eqref{P} has the mountain pass geometry we consider \eqref{posi_algum}. Also, since we also deal with constrained minimization, an autonomous version of \eqref{posi_algum_auto} is needed (see \cite{berlions}). 
%\item Since we are not interested in the sign of the solutions, we can assume alternatively to condition \eqref{posi_algum} that 
%\begin{align}\label{posi_algum_neg} \tag{$f_{4-}$}
%\begin{aligned}
%& \exists \, (x_0,t_0) \in \mathbb{R}^N \times \mathbb{R}_+ \text{  such that} \\
%& |B_R|\inf _{B_R(x_0)} F(x,t_0) + |B_{R+1}\setminus B_R| \inf_{(x,t) \in B_{R+1}(x_0)\setminus B_R(x_0) \times [t_0,0]} F(x,t) >0
%\end{aligned} &&
%\end{align}
%We empathize that one can replace \eqref{posi_algum} by \eqref{posi_algum_neg} in our results. For simplicity we assume that $t_0>0$ and only consider assumption \eqref{posi_algum}.
\item The asymptotic additivity \eqref{f_extra} ensure the convergence of the functional $\Phi$ under the weak profile decomposition for bounded sequences in $\mathcal{D}^{s,2}(\mathbb{R}^N)$ described in Theorem \ref{teo_tinta_frac}.
\item The smoothness condition $f(t)\in C^1(\mathbb{R})$ is the natural hypothesis used in the literature to prove that weak solutions of Eq. \eqref{P} satisfies a Pohozaev type identity.
\item Once the limits in \eqref{f_5} exist, to obtain compactness of Palais-Smale sequences at the minimax levels we need to require the additional conditions over the minimax levels $c_0, c_+, c_{-}$ given in assumption \eqref{suficient}. In fact, we do not believe that it is possible, in general, to achieve the compactness described in Theorem \ref{principal} without these conditions.  We mention that this kind of approach was introduced by P.-L. Lions in \cite{lionscompcase1,lionscompcase2,lionslimitcase1,lionslimitcase2}.
\item Observe also that the approach to obtain concentration-compactness for the autonomous case $f(x,t)=f(t)$ needs to be different since in this case, $f(t)$ does not satisfies \eqref{suficient}.
\item We also consider the case when \eqref{suficient} does not hold. Precisely, when it is allowed $c(I) = c(I_\kappa),$ for some $\kappa=0,+,-.$ In this case, the concentration-compactness argument at the mountain pass level cannot be used. We apply \cite[Theorem~2.3]{lins} to overcome this difficulty and prove existence of solution at the mountain pass level.
\item The limits $F _{\pm}$ fulfill \eqref{selfsimilar}. Thus, functions $f(x,t)$ that satisfies \eqref{f_5} could be seen as being asymptotically self-similar at $\pm \infty$ (see Remark \ref{they_are_ss}).
\item In Lemma~\ref{path} we proved that $\Gamma_{\kappa} \neq \emptyset$ is equivalent to: $\exists \, t_\kappa $ such that
$ F_\kappa (t_\kappa ) >0 $. Consequently \eqref{posi_algum_auto} is the most general assumption to ensure that $c(I)$ given in \eqref{minimax} is well defined (possible valuing $\pm \infty$).
\item One can consider the closed subspace of $\mathcal{D}^{s,2}(\mathbb{R}^N)$ consisting of radial functions, that is,
\begin{equation*}
\mathcal{D}^{s,2}_\rad (\mathbb{R}^N) = \left\lbrace u \in \mathcal{D}^{s,2}(\mathbb{R}^N) : u(x) = u(y),\text{ provided that }|x|=|y|\right\rbrace ,
\end{equation*}
to obtain more compactness. Precisely, we have that $w^{(n)}$ belongs to $\mathcal{D}^{s,2}_\rad (\mathbb{R}^N)$ with $w^{(n)} = 0,$ for all $n \in \mathbb{N}_0.$ The proof of this fact follows the same arguments as in \cite[Proposition 5.1]{tintabook}. Consequently our main results hold, replacing $\mathcal{D}^{s,2}(\mathbb{R}^N)$ by $\mathcal{D}^{s,2}_\rad (\mathbb{R}^N),$ and assuming that $f(x,t)=f(|x|,t)$ is radial in $ x$ instead of the existence of the asymptote $f_0(t).$
\end{enumerate}
\end{remark}
\begin{remark}\label{remark_positivo}
We have that $\mathcal{G} \neq \emptyset$ and $\mathcal{S} _l>0,$ provided \eqref{f_1}, \eqref{f_2} and \eqref{posi_algum_auto} holds. In fact, this follow as in \cite[Lemma 2.6 and Remark 2.8]{pala_vald_sch}. Indeed, let $v _R \in C^\infty _0 (\mathbb{R}),\ R>0,$ such that $0\leq v _R(t)\leq t_0$ and
\begin{align*}
v _R(t)=
\left\{\begin{array}{llc}
t_0,&\quad\text{if }|t|\leq R,\\
0,&\quad \text{if } |t|>R+1.
\end{array}
\right.
\end{align*}
For all $x \in \mathbb{R}^N,$ taking $\varphi_R(x) := v_R(|x|),$ we have $\varphi_R \in \mathcal{D}^{s,2} (\mathbb{R}^N ).$ Moreover, 
\begin{align*}
\int_{\mathbb{R}^N} F(\varphi_R ) \dx &= \int _{B_R (x_0)} F(t_0) \dx + \int_{B_{R+1}(x_0)\setminus B_R(x_0)} F(\varphi _R ) \dx\\
&\geq F(t_0)|B_R|-|B_{R+1}\setminus B_R|\left(\max_{t \in [0,t_0]} |F(t)|\right).
\end{align*}
Thus there exist two positive constants $C_1$ and $C_2$ such that
\begin{equation*}
\int_{\mathbb{R}^N} F(\varphi_R ) \dx\geq C_1 R^N - C_2 R^{N-1}>0,
\end{equation*}
provided that $R$ is taken large enough. Taking a suitable $\sigma >0$, we may conclude that $\Phi (\varphi_R (\cdot / \sigma ))=1.$ The case where $t_0<0$ is analogous.
\end{remark}
\begin{example}
Typical examples (see Section \ref{self-similar_functions} and the proof of Lemma \ref{basic_prop_ss}) of a function satisfying \eqref{f_1}--\eqref{suficient} are given by
\begin{enumerate}[label=(\roman*)]
\item$f(x,t) = b(x) |t| ^{2_s ^\ast-2}t,$ where $b(x) \in C (\mathbb{R}^N)$ with 
\begin{equation}\label{Cond_b}
b(x)>b(0) = \inf _{x \in \mathbb{R}^N} b(x) = \lim _{|x| \rightarrow \infty}b(x),
\end{equation}
and $b(0)>0.$
\item $f(x,t) = \exp\{ b(x)( \sin (\ln |t|) +2 )  \} (b(x) \cos (\ln |t|) + 2^{\ast} _s) |t| ^{2^{\ast} _s -2}t,$ with $f(x,0):=0;$
\end{enumerate}
where $b(x)\in C(\mathbb{R})$ satisfies \eqref{Cond_b}, $b(0) = 0$ and moreover
\[ 
\sup _{x \in \mathbb{R}^N} b(x)<2_ s ^\ast - \sigma,\quad \mbox{for some}\quad \sigma \in (2,2_s ^\ast).
\]
The primitive is given by $F(x,t)=\exp\{b(x)(\sin (\ln |t|)+2)\}|t|^{2_s ^\ast}.$
\end{example}
\begin{example}
The function $f(t)= (2_s ^\ast\cos(\ln|t|) -\sin(\ln|t|)) |t|^{2^\ast_s -2} t,$  $f(0) : = 0,$ satisfies the assumptions of Theorems \ref{primeira}--\ref{minimizacao}.
\end{example}
\section{Preliminaries}\label{Preliminaries}
\subsection{Fractional Sobolev spaces}
Let $0<s<N/2,$ by Placherel Theorem, we have
\begin{equation*}
[u]_s^2 = \int _{\mathbb{R}^N} |(-\Delta)^{s/2}u|^2 \dx,\quad\forall \, u \in C^\infty _0 (\mathbb{R}^N).
\end{equation*}
Thus, the space $\mathcal{D}^{s,2} (\mathbb{R}^N)$ is well defined with continuous embedding 
\begin{equation}\label{embds}
\mathcal{D}^{s,2}(\mathbb{R}^N) \hookrightarrow L ^{2^{\ast } _s} (\mathbb{R}^N) \quad \mbox{for} \quad 0 < s < N/2,
\end{equation}
in view of the well know inequality
\begin{equation*}
\int _{\mathbb{R}^N} |u| ^{2^\ast _s} \dx \leq \mathcal{K}_\ast \left(\int _{\mathbb{R}^N} |\xi|^{2s} |\mathscr{F}u|^2  \dxi \right)^{2_s^\ast /2},\quad \forall \, u \in C^\infty _0 (\mathbb{R}^N),
\end{equation*}
where
\begin{equation*}
\mathcal{K}_\ast =\left[  2^{-2s}\frac{\Gamma(\frac{N-2s}{2})}{\Gamma(\frac{N+2s}{2})}\left(\frac{\Gamma(N)}{\Gamma(N/2)}\right)^{2s/N} \right] ^{2_s ^\ast/2}
\end{equation*}
and $\mathscr{F}u$ is defined in \eqref{fourierdef}. Moreover, $\mathcal{D}^{s,2}(\mathbb{R}^N)$ as a separable Hilbert space when endowed with the inner product
\begin{equation*}
(u,v) = \int _{\mathbb{R}^N} (-\Delta )^{s/2} u (-\Delta )^{s/2 } v \dx, \quad  \forall \, u,v\in \mathcal{D}^{s,2} (\mathbb{R}^N),
\end{equation*}
as well the characterization 
\begin{equation*}
\mathcal{D}^{s,2}(\mathbb{R}^N) = \left\{ u \in L ^{2^{\ast } _s } (\mathbb{R}^N ) : \ (-\Delta )^{s/2} u \in  L^2 (\mathbb{R} ^N) \right\} =\left\{u \in L ^{2^{\ast } _s } (\mathbb{R}^N ) : \ | \cdot | ^s \mathscr{F} u \in L^2 (\mathbb{R} ^N) \right\}.
\end{equation*}
For $\Omega \subset \mathbb{R}^N$ open set and  $0<s<1,$ the inhomogeneous fractional Sobolev space is defined as
\begin{equation}\label{gravi}
H ^s (\Omega)=\left\lbrace u \in L^2 (\Omega) :  \int _{\Omega} \int _{\Omega} \frac{\left| u(x) - u(y) \right|^2}{|x-y|^{N + 2s}}dxdy <\infty \right\rbrace,
\end{equation}
with the norm
\begin{equation*}
\| u \| _{H ^s (\Omega )} ^2 := \int _{\Omega} u^2\dx + \int _{\Omega} \int _{\Omega} \frac{\left| u(x) - u(y) \right|^2}{|x-y|^{N + 2s}}\dxdy .
\end{equation*}
When $0<s<1,$ by \cite[Proposition 3.4]{hitchhiker},
\begin{equation*}
\|u\|^2 = \frac{C(N,s)}{2} \int _{\mathbb{R}^N} \int _{\mathbb{R}^N} \frac{ |u(x) - u(y) |^2}{|x-y|^{N + 2s}}\dxdy,\quad\forall \, u \in \mathcal{D}^{s,2} (\mathbb{R}^N),
\end{equation*}
for some positive constant $C(N,s).$  Thus, when $\Omega = \mathbb{R}^N,$ we have
\begin{equation}\label{manga}
H ^s (\mathbb{R} ^N) = \left\lbrace u \in L^2 (\mathbb{R}^N) : \ | \cdot | ^s \mathscr{F} u \in L^2 (\mathbb{R} ^N) \right\rbrace =\left\lbrace u \in L^2 (\mathbb{R}^N) : (-\Delta )^{s/2} u \in  L^2 (\mathbb{R} ^N) \right\rbrace.
\end{equation}
Turns out that definition of $H ^s (\mathbb{R} ^N)$  given in \eqref{manga} it is more appropriated for the general case $s\geq0,$ than definition \eqref{gravi}, because for $s \geq 1,$ the integral in \eqref{gravi} is finite if and only if $u$ is constant (see \cite[Proposition 2]{constant}). Moreover, we have the continuous embedding
\begin{equation}\label{Hembd}
H^s (\Omega) \hookrightarrow L^p(\Omega),\quad 2 \leq p \leq 2_s ^\ast,\quad\text{for}\quad 0 < s < N/2,
\end{equation}
and the following compact embedding (see \cite[Section 7]{hitchhiker}),
\begin{equation}\label{comp_loc}
H^s(\mathbb{R}^N)\hookrightarrow L _{\loca}^p (\mathbb{R}^N),\quad 1 \leq p < 2_s ^\ast,\quad\text{for}\quad 0 < s < 1.
\end{equation}
Consequently, we have that every bounded sequence in $\mathcal{D}^{s,2}(\mathbb{R}^N)$ has a subsequence that converges almost everywhere and weakly in $\mathcal{D}^{s,2}(\mathbb{R}^N),$ for $0<s<\min\{1,N/2\}$. Let $\mathscr{S}_0$ the subspace of $\mathscr{S}$ consisting in all function $u$ such that $\mathscr{F}u \in C^\infty _0 (\mathbb{R}^N\setminus \{0\} ).$ In this case, $(-\Delta) ^s u \in \mathscr{S}.$ We finish this section emphasizing that the Plancherel Theorem also gives the next identity, which it will be used several times throughout this paper
\begin{equation}\label{formula}
\int _{\mathbb{R}^N} (-\Delta )^{s/2} u (-\Delta )^{s/2 } v \dx = \int _{\mathbb{R}^N} (-\Delta )^{s} u  v \dx, \; \forall \, u,v \in \mathscr{S}_0.
\end{equation}
\subsection{Local regularity and Pohozaev Identity}\label{pohozaevs}
We are in the position to prove that weak solutions of autonomous form of Eq.~\eqref{P} are $C^1(\mathbb{R}^N)$ and satisfies the Pohozaev identity
\begin{equation}\label{pohozaev}
\int _{\mathbb{R}^N} \big|(-\Delta)^{s/2} u(x) \big|^2 \dx = \frac{2N}{N-2s} \int _{\mathbb{R}^N} F(u(x)) \dx,
\end{equation}
under suitable assumptions on $f(t)$ (see Proposition \ref{prop_pohozaev} for the precise statement). We refer to \cite{chang-wang}, where the identity was studied for solutions in $H^s (\mathbb{R}^N)$ and when $f(t)$ satisfy a fractional version of the H.~Berestycki and P.-L. Lions assumptions. The main idea for that, it is to use the so called Caffarelli-Silvestre extension (see \cite{caf_silv} for more details) which transform the autonomous non-local Eq.~\eqref{P} in a local one and use recent regularity results to develop the resultant expression in a such way to apply the argument of \cite[Section 2]{berlions}. Our approach is in some way different from the usual one. Although we continue using Caffarelli-Silvestre extension (also know as harmonic extension), by the results of \cite{moustapha} and \cite{frac_niremberg}, we can derive a local regularity for weak solutions in $\mathcal{D}^{s,2} (\mathbb{R}^N)$ in a more suitable way to get the desired identity by applying a truncation argument.

Next for the readers convenience we introduce the harmonic extension following \cite[Section 2]{frac_niremberg} and for that we begin  describing a class of weight Sobolev spaces suitable to work with this harmonic extension. First, observe that, for any $0<s<1$, the function $z=(x,y) \mapsto |y|^{1-2s}$ belongs to the Muckenhoupt class $\mathcal{A}_2$ of weights in $\mathbb{R}^{N+1},$ that is
\begin{equation*}
\left( \frac{1}{|B|} \int _{B} |y|^{1-2s} \dxdy \right) \left(\frac{1}{|B|}  \int _{B}|y|^{2s-1 }\dxdy\right) \leq C, \;  \forall \text{ ball } B\subset \mathbb{R}^{N+1}.
\end{equation*}
More details can be found in \cite{fabes}. Let $Q$ be a open set in $\mathbb{R}^{N+1},$ we consider $L^2 (Q, |y| ^{1-2s})$ as the Banach space of the Lebesgue measurable functions $v$ defined in $Q$ such that
\begin{equation*}
\|v\| _{L^2 (Q, |y| ^{1-2s})} = \left(\int _{Q }|y|^{1-2s} v^2 \dxdy \right)^{1/2}<\infty.
\end{equation*}
We also consider the space $H^1 (Q , |y| ^{1-2s} )$ of the functions $w$ in $L^2 (Q, |y| ^{1-2s})$ such that its weak derivatives $w_{z_i}$ exists and belongs to $L^2 (Q, |y| ^{1-2s})$ for $\ i=1,\ldots,N+1.$ It is easy to see that $H^1 (Q , |y| ^{1-2s} )$ is a Hilbert space with inner product
\begin{equation*}
( v_1,v_2 ) _{H^1 (Q, |y|^{1-2s} ) } = \int _{Q} |y|^{1-2s} \left( \left\langle \nabla v_1 , \nabla v_2\right \rangle +v_1 v_2  \right) \dxdy,
\end{equation*}
and the induced norm
\begin{equation*}
\|v\| _{H^1 (Q, |y|^{1-2s} ) } = \left( \int _{Q} |y|^{1-2s} \big|\nabla v (x,y) \big| ^2 + |y|^{1-2s} v ^2 (x,y) \dxdy\right) ^{1/2}.
\end{equation*}
We call attention to the fact that the space of smooth functions $C^\infty (Q)\cap H^1 (Q , |y| ^{1-2s} )$ is dense in the weight Sobolev space $H^1 (Q , |y| ^{1-2s} )$ (see \cite{turesson} for further details).

Regarding the space $H^1 (Q , y ^{1-2s} )$ with $Q=\Omega \times (0,R),$ where $\Omega \subset \mathbb{R}^N$ is a domain  with Lipschitz boundary, it is well know the existence of a well-defined trace operator $$t_r : H^1 (Q , y ^{1-2s} ) \rightarrow H^s (\Omega)$$ with
\begin{equation*}
\| t_r (v) \| _{H^s (\Omega )} \leq C \| v \| _{H^1 (Q, y^{1-2s} )}, \; \forall \,  v \in H^1 (Q, y^{1-2s} ),
\end{equation*}
where $C>0,$ depends only on $N,$ $s$ and $\Omega$ (see also \cite{nekvinda}). Moreover, by the continuous embedding $H^s (\Omega) \hookrightarrow L^{2 ^\ast _s} (\Omega),$ we have
\begin{equation}\label{tracinho}
\| t_r (v) \| _{L ^{2^\ast _s} (\Omega)} \leq C \| v \| _{H^1 (Q, y^{1-2s} )}, \, \forall v \in H^1 (Q, y^{1-2s} ).
\end{equation} 
Let
\begin{equation*}
P _s(x,y) = \beta(N,s)\frac{y^{2s}}{\left(|x|^2 + y^2 \right)^{\frac{N+2s}{2}} },
\end{equation*}
where $\beta(N,s)$ is such that $\int _{\mathbb{R^N}} P _s (x,1) \dx =1$ and $0<s<\min\{1,N/2\}.$ Considering the standard notation
 $$\mathbb{R}^{N+1} _+=\{(x,y)\in \mathbb{R}^{N+1}:y>0 \},$$
for $u \in \mathcal{D}^{s,2} (\mathbb{R}^N)$ let us set the $s-$harmonic extension of $u$,
\begin{equation*}
w(x,y)= E_s (u) (x,y) := \int _{\mathbb{R}^N} P _s (x- \xi,y) u (\xi) \dxi,\quad (x,y) \in \mathbb{R}_+ ^{N+1}.
\end{equation*}
Then, for any compact subset $K$ of $\overline{\mathbb{R}^{N+1} _+},$ we have $w \in L^2 (K, y^{1-2s})$, $\nabla w \in L ^2 (\mathbb{R}^{N+1} _+, y^{1-2s})$ and  $w \in C^{\infty} (\mathbb{R}^{N+1} _+).$ Moreover, $w$ satisfies
\begin{equation}\label{sol_div}
\left\{
\begin{aligned}
\dive (y^{1-2s} \nabla w ) & = 0,  & \text{in } & \mathbb{R}^{N+1} _+,  \\
-\lim _{y \rightarrow 0 ^+} y^{1-2s} w_y (x,y) & = \kappa _s (- \Delta) ^s u (x)   & \text{in } & \mathbb{R}^{N},\\
\| \nabla w \| ^2 _{ L ^2 (\mathbb{R}^{N+1} _+, y^{1-2s}) } &=\kappa _s \|u\|^2,
\end{aligned}
\right.
\end{equation}
where we understand \eqref{sol_div} in the distribution sense, where $\kappa _s = 2^{1-2s}  \Gamma (1-s) / \Gamma (s),$ and $\Gamma$ is the gamma function. Precisely,
\begin{equation*}
\int _{B_R ^+} y^{1-2s} \left\langle  \nabla w , \nabla \varphi \right\rangle \dxdy = \kappa _s \int_{B_R ^N} (-\Delta) ^{s/2} u (-\Delta) ^{s/2} (t_r \varphi) \dx, \; \forall  \varphi \in C ^\infty _0 (B ^+_R \cup B^N _R),
\end{equation*}
where for $R>0$,
\[
\left\{
\begin{aligned}
& B_R = \{ z=(x,y) \in \mathbb{R}^{N+1}  :|z|^2<R^2 \},\\
& B_R ^+ = B_R \cap \mathbb{R}^{N+1} _+\text{ and}&\\
& B ^N _R = \{ z=(x,y) \in \mathbb{R}^{N+1} _+  :|z|^2<R^2, \  y = 0 \}. &
\end{aligned}
\right.
\]
More generally, given $g(t) \in C (\mathbb{R} )$ such that
\begin{equation}\label{ge_grow}
|g(t)| \leq C |t|^{2^\ast _s-1},\quad\forall t \in \mathbb{R},
\end{equation}
we say that a function $v \in H^1 (B ^+ _R , y ^{1-2s} )$ is a weak solution of the problem
\begin{equation}\label{weak_final}
\left\{
\begin{aligned}
\dive (y^{1-2s} \nabla v )  = &  0  \quad  & \text{in} \quad  B^+_R,\\
-\lim _{y \rightarrow 0 ^+} y^{1-2s} v_y (x,y)  = & \kappa _s g(t_r(v)(x) )) \quad & \text{in} \quad B^N _R,
\end{aligned}
\right.
\end{equation}
if, for all $\varphi \in C ^\infty _0 (B ^+_R \cup B^N _R)$, we have
\begin{equation}\label{defi_fraca}
\int _{B_R ^+} y^{1-2s} \left\langle  \nabla v , \nabla \varphi \right\rangle \dxdy = \kappa _s \int_{B_R ^N} g(t_r (v ) ) t_r ( \varphi ) \dx.
\end{equation}
Thus $w=E_s (u)$ is a weak solution of \eqref{weak_final} with $g(t)=f(t)$ if, and only if, $u$ is a weak solution of Eq.~\eqref{P}. In the first result of this section we derive, following the regularity results of \cite{frac_niremberg}, sufficient conditions over the harmonic extension which ensures the validity of identity \eqref{pohozaev}. 
\begin{proposition}\label{main_reg}
Let $v \in H^1 (B ^+ _R , y ^{1-2s} )$ be a weak solution of \eqref{weak_final}. Suppose that $g \in C^1 (\mathbb{R} )$ satisfies \eqref{ge_grow}. If $t_r (v) \in L ^p _{\loca} (B ^N _r ),$ for some $p >2^\ast _s,$ then for any
$R>0$ there exists  $0<y_0,\ r<R$ with $B ^N _r \times [0,y_0] \subset B ^+ _R,$ and $\alpha \in (0,1),$ such that
\begin{equation}\label{regularity}
v,\ \nabla _x v,\ y^{1-2s}v_y \in C^{0, \alpha} (B ^N _r \times [0, y_0 ]),
\end{equation}
where $\nabla _x v = (v_{x_1} , \ldots , v_{x_N}).$
\end{proposition}
\begin{proof}
In fact, since
\begin{equation*}
\frac{g(t_r v)}{1+|t_r v|}\in L _{\loca} ^q (\mathbb{R}^N),\quad \forall \; N/2s<q<p/(2^\ast_s -2), 
\end{equation*}
we can see that \eqref{regularity} follows taking 
\begin{equation*}
g(t_r v) = \frac{g(t_r v)}{1+|t_r v|} \sgn (t_r v) t_r v + \frac{g(t_r v )}{1+|t_r v|},
\end{equation*}
and proceeding analogously to the proof of \cite[Proposition 2.19]{frac_niremberg},  by applying \cite[Proposition 2.6. Proposition 2.13, Theorem 2.14 and Lemma 2.18]{frac_niremberg}.
\end{proof}

In order to apply Proposition~\ref{main_reg} we need to prove a Brezis-Kato type result (see \cite{breziskato}) for solutions of Eq.~\eqref{P}. Although a similar result can be found in \cite[Lemma 3.5]{moustapha}, the absence of singularity in Eq. \eqref{P} allows us to obtain a simpler proof. To achieve that, we strongly rely in the following lemmas, which enable us to proceed as in \cite{breziskato} (cf. \cite[Proposition 5.1]{barrios} or \cite[Theorem 1.2]{aliang}).

\begin{lemmaletter}\cite[Theorem~1.3]{fabes}\label{fabes}
For any $R>0,$ there exists $\sigma>1$ and $C_R>0$ depending on $R,$ such that
\begin{equation*}
\left( \int _{B_R} |y|^{1-2s} |v|^{2\sigma} \dxdy\right) ^{1/ \sigma}\leq C_R \int _{B_R} |y|^{1-2s} |\nabla v| ^2 \dxdy, \; \forall  v \in C^\infty _0 (B_R).
\end{equation*}
\end{lemmaletter}
\begin{lemmaletter}\cite[Lemma~2.6]{moustapha} \label{fall}
Let $\xi \in C (\mathbb{R}^{N+1})$ such that $\xi (z) =0$ for all $|z| \geq R.$ There exist $C>0$ such that
\begin{equation*}
\left( \int_{B^N _R} |v \xi | ^{2^\ast _s} \dxdy \right) ^{2/2_s ^\ast } \leq C \int _{B^+_R} y ^{1-2s} |\nabla (v \xi )| ^2 \dxdy, \; \forall  v \in H^1(B^+_R , y^{1-2s} ).
\end{equation*}
\end{lemmaletter}
\begin{proposition}\label{poho_prop}
Assume that condition \eqref{f_2} holds. Let $u \in \mathcal{D}^{s,2} (\mathbb{R}^N)$ be a weak solution of Eq.~\eqref{P} for the autonomous case, then $u \in L_{\loca} ^p (\mathbb{R}^N),$ for all $p\geq1.$
\end{proposition}
\begin{proof}
Let $w = E_s(u)$ and $\xi \in C_0 ^\infty (\mathbb{R}^{N+1}:[0,1])$ such that
\begin{equation*}
\xi(z)= \left\{
\begin{aligned}
&1, \quad\text{if }|z| <R/2\\
&0,\quad\text{if }  |z|\geq R
\end{aligned}
\right.
\qquad\text{ and}\qquad|\nabla \xi (z)| \leq C \quad \forall  z \in \mathbb{R}^{N+1},
\end{equation*}
for some $C>0.$ Since the map $ t \mapsto t \min \{|t| ^\beta ,L \},$ $\beta ,\ L>0,$ is Lipschitz in $\mathbb{R},$ considering $w_{\beta ,L} := \min \{ |w| ^\beta , L\}$ we have $ww_{\beta ,L}  \in H^1 (B ^+ _R , y ^{1-2s} ),$ consequently using inequality \eqref{tracinho} in a density argument one can see that $ ww_{\beta ,L} ^2 \xi ^2$ can be taken as a test function in definition \eqref{defi_fraca}. The main idea is to get the estimate
\begin{equation}\label{bk0}
\int _{B _R ^+} y ^{1-2s} |\nabla (w w_{\beta ,L} \xi)|^2 \dxdy \leq C,
\end{equation}
for a suitable $\beta$ and $C>0$ which does not depend on $L.$ The next step is to use Fatou Lemma and Lemma \ref{fall} to obtain
\begin{equation*}
\int _{B^N _R} |u| ^{(\beta+1 )2^\ast _s} \dx \leq C.
\end{equation*}
This leads to a iteration procedure in $\beta$ which implies in $u \in L ^p (B^N _R)$ for all $p>1.$
To do so, we start taking
\begin{equation*}
a(x) := \frac{|f(u)|}{1+|u|} \in L ^{N/2s}_{\loca} ( \mathbb{R}^N),
\end{equation*}
which implies
\begin{equation}\label{bk1}
\int _{B ^+_R} y^{1-2s} \left\langle \nabla w , \nabla (w w_{\beta ,L} ^2 \xi ^2 )\right\rangle \dxdy \leq 2 \kappa _s \int _{B ^N _R } a(x) (1 + u ^2) u _{\beta , L} ^2 \xi^2 \dx,
\end{equation}
 where we used that $(1+t)t \leq 2(1+t^2),\ t>0$ and $t_r (w w_{\beta ,L} ^2 \xi ^2) = u u_{\beta ,L} ^2 \xi(\cdot,0) ^2.$ We now compute the left side of the inequality \eqref{bk1} and use the following identity
 \begin{equation*}
 w \left\langle \nabla w , \nabla (|w| ^{2 \beta } ) \right\rangle = \frac{\beta}{2} |w| ^{2 (\beta -1 )} |\nabla (w ^2 ) | ^2,
 \end{equation*}
 to conclude
 \begin{multline}\label{bk2}
\int _{B^+ _R} y^{1-2s} \min \{|w|^{2 \beta}  , L ^2 \} |\nabla w| ^2 \xi ^2 \dxdy + \frac{\beta}{2}\int _{\{ |w|^{2 \beta } \leq L^2 \} \cap B^+ _R } y^{1-2s} |w| ^{2 (\beta -1 )} |\nabla (w ^2 ) | ^2 \xi ^2 \dxdy \\ \leq 2 \kappa _s \int _{B ^N _R } a(x) (1 + u ^2) u _{\beta , L} ^2 \xi ^2 \dx - 2 \int _{B ^+ _R} y ^{1-2s} w \min \{ |w| ^{2\beta} , L^2\} \xi \left\langle \nabla w , \nabla \xi \right\rangle \dxdy.
 \end{multline}
 Using the Cauchy inequality (with $\varepsilon = 1/4$) we have
 \begin{multline}\label{bk3}
 - 2 \int _{B ^+ _R} y ^{1-2s} w \min \{ |w| ^{2\beta} , L^2\} \xi \left\langle \nabla w , \nabla \xi \right\rangle \dxdy \\ \leq \frac{1}{2} \int _{B ^+ _R} y^{1-2s} \min \{ |w| ^{2\beta} , L^2\} |\nabla w| ^2 \xi ^2 \dxdy + C \int _{B ^+ _R} y^{1-2s} w^2 \min \{ |w| ^{2\beta} , L^2\} |\nabla \xi |^2 \dxdy,
\end{multline}
where $C>0$ is independent of $L.$ From replacing \eqref{bk3} in \eqref{bk2}, we obtain
\begin{multline}\label{bk4}
\frac{1}{2} \int _{B ^+ _R} y^{1-2s} \min \{ |w| ^{2\beta} , L^2\} |\nabla w| ^2 \xi ^2 \dxdy + \frac{\beta}{2}\int _{\{ |w|^{2 \beta } \leq L^2 \} \cap B^+ _R } y^{1-2s} |w| ^{2 (\beta -1 )} |\nabla (w ^2 ) | ^2 \xi ^2 \dxdy \\ \leq C \int _{B ^+ _R} y^{1-2s} w^2 \min \{ |w| ^{2\beta} , L^2\} |\nabla \xi |^2 \dxdy +2 \kappa _s \int _{B ^N _R } a(x) (1 + u ^2) u _{\beta , L} ^2 \xi ^2 \dx.
\end{multline}
Now using
\begin{equation*}
\beta ^2 |w| ^{2(\beta -1)} \left| \nabla (w ^2 )\right| ^2 = 4 w ^2 \left| \nabla (|w| ^\beta )\right| ^2,
\end{equation*}
together with inequality \eqref{bk4}, we can finally estimate \eqref{bk0},
\begin{multline}\label{bk5}
\int _{B _R ^+} y ^{1-2s} |\nabla (w w_{\beta ,L} \xi)|^2 \dxdy\\  \leq C \int _{B ^+ _R} y^{1-2s} w^2 \min \{ |w| ^{2\beta} , L^2\} |\nabla \xi |^2 \dxdy +2 \kappa _s \int _{B ^N _R } a(x) (1 + u ^2) u _{\beta , L} ^2 \xi ^2 \dx.
\end{multline}
It remains to estimate the last two terms in \eqref{bk5}. Assuming $|u|^{\beta +1} \in L ^2 (B^N _R),$ we get
\begin{align*}
\int _{B ^N _R }  a(x)u ^2u _{\beta , L} ^2 \xi^2 \dx &\leq L_0\int _{B ^N _R} |u| ^{2(\beta +1)}\xi ^2 \dx+ \int _{B ^N _R\cap \{a(x)\geq L_0\}}a(x)u ^2 u _{\beta , L} ^2 \xi^2 \dx \\&\leq C_1 L_0 + \tilde{C}_1\varepsilon (L_0 ) \left(\int _{B _R ^+} y ^{1-2s} |\nabla (w w_{\beta ,L} \xi)|^2 \dxdy\right)^{2/2^\ast _s},
\end{align*}
where
\begin{equation*}
\varepsilon (L_0 ) := \left(\int _{ \{a(x) \geq L_0 \} } a ^{N/2s}(x) \dx \right) ^{2s/N} \rightarrow 0,\text{ as }L_0 \rightarrow \infty. 
\end{equation*}
By the same calculation and using $\min \{ |t| ^\beta ,L \} \leq |t|\min \{ |t| ^\beta ,L \} +1, L>1,$ we obtain
\begin{multline*}
\int _{B ^N _R }  a(x) u _{\beta , L} ^2 \xi ^2 \dx \\ \leq C_2 L_0  + \tilde{C}_2\varepsilon (L_0 ) \left[ \left( \int_{B^+_R } y ^{1-2s} |\nabla (w w_{\beta ,L} \xi)|^2 \dxdy \right) ^{2/2^\ast _s} + \left( \int _{B^N _R} |\xi | ^{2_s ^\ast} \dx \right) ^{2/2^\ast _s} \right],
\end{multline*}
Thus, we can take $L_0$ large enough such that 
\begin{equation*}
\int _{B _R ^+} y ^{1-2s} |\nabla (w w_{\beta ,L} \xi)|^2 \dxdy \leq C_3 \int _{B ^+ _R} y^{1-2s} w^2 \min \{ |w| ^{2\beta} , L^2\} |\nabla \xi |^2 \dxdy.
\end{equation*}
Finally, assume that $\beta+1 \leq \sigma$, where $\sigma$ is given in Lemma \ref{fabes}. Using the operator extension by reflection $\mathcal{R}: H^1 (B^+_R,y^{1-2s} )\rightarrow H^1 (B_R,|y|^{1-2s} )$ given by
\begin{equation*}
\mathcal{R}(w)(x,y)=\left\{
\begin{aligned}
&w(x,y),\quad& \text{if }y>0,\\
&w(x,-y),\quad& \text{if }y \leq 0,
\end{aligned}
\right.
\end{equation*}
(see for instance \cite[Section 4]{caf_silv}), we may apply Lemma \ref{fabes} for an appropriated sequence of functions in $C^\infty _0 (\mathbb{R}^{N+1}),$ converging to $\mathcal{R}(w)$ in $H^1 (B_R,|y|^{1-2s})$ to get
\begin{equation*}
\int _{B ^+ _R} y^{1-2s} w^2 \min \{ |w| ^{2\beta} , L^2\} |\nabla \xi |^2 \dxdy \leq C_4\int_{B_R} |y|^{1-2s} |\nabla (\mathcal{R}(w))|^2 \dxdy\leq C_5.
\end{equation*}
We take $\beta=\beta_1=\min \{2^\ast _s/2, \sigma \} -1$ and $\beta_{i+1} = \min\{2^{\ast } _s/2,\sigma \} (2 ^\ast _s/2)^i - 1,$ $i=0,1,\ldots,$ to obtain that $u\in L ^{\beta_{i+1} } _{\loca} (\mathbb{R}^N).$
\end{proof} Summing up all the previous results we can finally conclude the validity of identity \eqref{pohozaev} and the desired local regularity.
\begin{proposition}\label{prop_pohozaev}
If $f(t)\in C^1 (\mathbb{R})$ and satisfies \eqref{f_2}, then every weak solution of Eq.~\eqref{P} for the autonomous case belongs to $C^1 (\mathbb{R}^N).$ Moreover, the Pohozaev identity \eqref{pohozaev} holds true.
\end{proposition}
\begin{proof}
Let $u \in \mathcal{D}^{s,2} (\mathbb{R}^N)$ be a weak solution of Eq.~\eqref{P} for the autonomous case with $f(t)$ satisfying \eqref{f_2}. Consider $w=E_s (u),$ then by Propositions \ref{main_reg}, $w$ possess the regularity \eqref{regularity}. In particular, $\nabla u=\nabla w(x,0) \in C(B^N _r)$ for any $r>0.$ Let $\xi \in C_0 ^\infty (\mathbb{R}:[0,1])$ such that
\begin{equation*}
\xi(t)=
\left\{\begin{array}{llc}
1, \quad\text{if } \quad t \in [-1,1]\\
0,\quad\text{if }  \quad |t|\geq 2
\end{array}
\right.
\qquad\text{ and}\qquad|\xi ' (t)| \leq C \quad  \forall  t\in \mathbb{R},
\end{equation*}
for some $C>0.$ For each $n\in \mathbb{N},$ define $\xi _n : \mathbb{R}^{N+1} \rightarrow \mathbb{R}$ by $\xi _n (z) = \xi (|z|^2/n^2).$ Then $\xi _n \in C_0 ^\infty (\mathbb{R}^{N+1})$ and verifies 
\begin{equation}\label{xin}
|\nabla \xi _n (z)| \leq C \qquad \text{and}\qquad |z||\nabla \xi _n (z)| \leq C\quad \forall z\in \mathbb{R}^{N+1},
\end{equation}
for some $C>0.$ Now observe that taking $w=E_s (u),$ 
\begin{multline}\label{contao}
\dive (y^{1-2s} \nabla w) \left\langle z , \nabla w \right\rangle \xi _n = \\
 \dive \left[y^{1-2s} \xi _n \left( \left\langle z , \nabla w \right\rangle \nabla w - \frac{|\nabla w|^2}{2} z\right)\right]  + \frac{N-2s}{2} y^{1-2s} |\nabla w|^2 \xi _n \\ + y^{1-2s} \frac{|\nabla w|^2}{2} \left\langle z , \nabla \xi _n  \right\rangle - y^{1-2s} \left\langle \nabla w , z \right\rangle\left\langle \nabla w , \nabla \xi _n \right\rangle .
\end{multline}
Given $\delta > 0$ we set
\[
\left\{
\begin{aligned}
& B_{n,\delta}= \{ z=(x,y) \in \mathbb{R}^{N+1} _+  :|z|^2<2n^2, \  y > \delta \} , & \\
& F^1_{n,\delta}= \{ z=(x,y) \in \mathbb{R}^{N+1} _+  :|z|^2<2n^2, \  y = \delta \} , & \\
& F^2_{n,\delta}= \{ z=(x,y) \in \mathbb{R}^{N+1} _+  :|x|^2 +y^2 = 2n^2, \  y > \delta \} . &
\end{aligned}
\right.
\]
Hence $\partial B _{n, \delta } =  F^1 _{n,\delta } \cup  F^2 _{n, \delta} .$ Let $\eta(z)=(0, \ldots,-1)$ be the unit outward normal vector of $B_{n,\delta}$ on $F^1_{n,\delta},$  since $\xi _n = 0$ on $ F^2 _{n, \delta},$ by condition \eqref{sol_div}, identity \eqref{contao} and the Divergence Theorem we get
\begin{align*}
0&=\int _{B_{n , \delta }} \dive (y^{1-2s} \nabla w) \left\langle z , \nabla w \right\rangle \xi _n \dxdy \\
&=\int _{F^1_{n,\delta}} y^{1-2s} \xi _n \left(\left\langle  z, \nabla w \right\rangle \left\langle  \nabla w , \eta \right\rangle   - \frac{|\nabla w|^2}{2} \left\langle   z, \eta  \right\rangle  \right)\dxdy + \theta _{n, \delta}\\
&=\int _{F^1_{n,\delta}} \xi _n \left\langle  x, \nabla _x w \right\rangle (-y ^{1-2s} w_y) \dx - \int _{F^1_{n,\delta}} y ^{1-2s} \xi _n |w_y| ^2 y \dx + \int _{F^1_{n,\delta}} y ^{1-2s} \xi _n \frac{|\nabla w|^2}{2} y \dx + \theta _{n, \delta}\\
&=I^1 _{n, \delta} + I^2 _{n, \delta} + I^3 _{n, \delta} + \theta _{n, \delta},
\end{align*}
where
\begin{equation*}
\theta _{n, \delta} = \int _{B_{n , \delta }} \frac{N-2s}{2} y^{1-2s} |\nabla w|^2 \xi _n + y^{1-2s} \frac{|\nabla w|^2}{2} \left\langle z , \nabla \xi _n  \right\rangle - y^{1-2s} \left\langle \nabla w , z \right\rangle\left\langle \nabla w , \nabla \xi _n \right\rangle \dxdy.
\end{equation*}
Using the same arguments as in \cite[proof of Theorem 3.7]{moustapha} we deduce that there exists a sequence $\delta _k \rightarrow 0$ such that
\begin{equation*}
I^2 _{n, \delta _k} + I^3 _{n, \delta _k} \rightarrow 0,\ \text{as } k \rightarrow \infty.
\end{equation*}
Some computations leads to
\begin{equation*}
\xi _n (x,0) \left\langle x , \nabla u \right\rangle f(u) = \dive (\xi _n (x,0) F(u)x)- F(u) \left\langle \nabla \xi _n (x,0) , x\right\rangle - \xi _n (x,0) F(u)N .
\end{equation*}
Setting 
$$
B^N _{\sqrt{2} n} = \left\lbrace  (x,y) \in \mathbb{R}^{N+1} : |x|^2 \leq R^2, \ y=0 \right\rbrace,
$$ 
by condition \eqref{sol_div}, the Divergence Theorem and the fact that $w$ satisfy \eqref{regularity}, we have
\begin{align*}
\lim_{k \rightarrow \infty} I^1_{n, \delta _k} &= \kappa _s \int _{B^N _{\sqrt{2} n}} \xi _n (x,0) \left\langle x , \nabla u \right\rangle f(u) \dx \\
&=\kappa _s\int _{B^N _{\sqrt{2} n}} \dive ( \xi _n (x,0) F(u)x ) ) - F(u) \left\langle \nabla \xi _n (x,0) , x\right\rangle - \xi _n (x,0) F(u)N \dx \\
&=-N \kappa _s \int _{B^N _{\sqrt{2} n} } \xi _n (x,0) F(u) dx - \kappa _s \int _{B^N _{\sqrt{2} n} } F(u) \left\langle \nabla \xi _n (x,0),x\right\rangle \dx. 
\end{align*}
Summing up, we have
\begin{align*}
0&=\lim_{k \rightarrow \infty} (I^1 _{n, \delta_k} + I^2 _{n, \delta_k} + I^3 _{n, \delta_k} + \theta _{n, \delta_k} )\\
&=-N \kappa _s \int _{B^N _{\sqrt{2} n} } \xi _n F(u) \dx - \kappa _s \int _{B^N _{\sqrt{2} n} } F(u) \left\langle \nabla \xi _n,x\right\rangle \dx\\
& \quad + \int _{B^+ _{\sqrt{2} n} }\frac{N-2s}{2} y^{1-2s} |\nabla w|^2 \xi _n + y^{1-2s} \frac{|\nabla w|^2}{2} \left\langle z , \nabla \xi _n \right\rangle - y^{1-2s} \left\langle \nabla w , z \right\rangle\left\langle \nabla w , \nabla \xi _n  \right\rangle \dxdy.
\end{align*}
Consequently taking $n \rightarrow \infty$ and using conditions \eqref{xin}, we conclude
\begin{equation}
\frac{N-2s}{2} \int _{\mathbb{R}^{N+1}_+} y ^{1-2s} |\nabla w| ^2 \dxdy = N \kappa _s \int _{\mathbb{R}^N} F(u) \dx,
\end{equation}
which together with condition \eqref{sol_div} implies \eqref{pohozaev}, and the proof is complete.
\end{proof}
\subsection{D-weak convergence and dislocation spaces}
As already mentioned, to achieve the decomposition described in Theorem \ref{teo_tinta_frac}, we follow the abstract approach of $D$-weak convergence and dislocation spaces developed in \cite{tintabook}. For the convenience of the reader we state the basic concepts without proofs, thus making our exposition self-contained. In this subsection $H$ denotes a separable infinite-dimensional Hilbert space.
\begin{definitionletter}\cite[Definition 3.1]{tintabook}
Let $D$ be a set of bounded linear operators such that for every $g \in D,$ $\inf _{u \in H, \| u \| = 1} \| gu \| > 0.$ We will say that the sequence $(u_k) \subset H$ converges to $u$ $D$-weakly in $H$, which we will denote as
\begin{equation*}
u_k \stackrel{D}{\rightharpoonup} u,\mbox{ in }H,
\end{equation*}
if for any sequence $(g_k) \subset D,$
\begin{equation*}\label{d3.1}
(g_{k}^{\ast}g_{k})^{-1}g_{k}^{\ast}(u_{k}-u)\rightharpoonup 0 \mbox{ in H.}
\end{equation*}
\end{definitionletter}
Let $(g_k)$ be a sequence of bounded linear operators in $H.$ It is commonly used in \cite{tintabook} the notation $g_k \rightharpoonup 0$ to indicate that $g_k u \rightharpoonup 0$ in $H$ for all $u \in H.$
\begin{definitionletter}\cite[Definition 3.2]{tintabook}
A set $D$ of bounded linear operators on $H$ is a set of dislocations if
\begin{eqnarray*}\label{d3.3}
0<\delta:=\inf_{{g\in D, \|u\|=1}}\|gu\|^{2}\leq \sup_{{g\in D, \|u\|=1}}^{}\|gu\|^{2}<\infty,\\
(u_k)\subset H, \ (g_k) \subset D , \  u_k \rightharpoonup  0 \mbox{ in }H\Rightarrow  g_{k}^\ast g_ku_k\rightharpoonup 0 \mbox{ in }H, \label{d3.4}
\end{eqnarray*}
and, whenever $(u_k) \subset H$ and $(g_k) , (h_k) \subset D$,
\begin{equation*}\label{d3.5}
h^{\ast} _k g_k \not \rightharpoonup 0, \ (g_{k}^{\ast}g_{k})^{-1}g^\ast_ku_k\rightharpoonup 0 \mbox{ in } H\Rightarrow (h_{k}^{\ast}h_{k})^{-1}h^\ast_{k} u_{k}\rightharpoonup 0 \mbox{ in } H.
\end{equation*}
The pair $(H,D)$ is called a dislocation space.
\end{definitionletter}
The next result give a sufficient condition to establish if a pair $(H,D)$ is a dislocation space.
\begin{propositionletter}\cite[Proposition~3.1]{tintabook}\label{prop3.1}
Let $D$ be a group (under the operator multiplication) of unitary operators 
$g :H \rightarrow H $, that is, $g^\ast = g^{-1}.$
 If
\begin{equation*}\label{3.7}
g_k \not \rightharpoonup 0 \mbox{ in }H,\ g_k \in D \Rightarrow g_k u  \mbox{ has a convergent subsequence, for all }u \in H,
\end{equation*}
then $(H,D)$ is dislocation space.
\end{propositionletter}
The next result provides a profile decomposition for bounded sequence in a suitable abstract Hilbert space, it is crucial to obtain the decomposition in Theorem \eqref{teo_tinta_frac}, and it can be seen as a generalization of the celebrated Banach-Alaoglu-Bourbaki Theorem.

\begin{theoremletter}\cite[Theorem~3.1]{tintabook} \label{teo_tinta}
	Let $ (H,D) $ be a dislocation space. If $ ( u_k) \subset H $ is a bounded sequence, then there exists a set $ \mathbb{N}_0 \subset \mathbb{N}, $ and sequences $ (w^{(n)}) _{n \in \mathbb{N}_0} \subset H, \; \; (g_k ^{(n)}) _{k \in \mathbb{N} } \subset D,\; \; \ g_k ^{(1)}=id,$ with $ n\in \mathbb{N}_0 $, such that for a subsequence of $ (u_k) $,
	\begin{flalign*}
	\qquad \qquad \qquad \qquad \qquad &\left(g_{k}^{(n)^\ast} g_{k}^{(n)}\right)^{-1}g_{k}^{(n)^\ast}
	u_k\rightharpoonup w^{(n)} \mbox { in H,} &\\ 
	\qquad \qquad \qquad \qquad \qquad &g_{k}^{(n)^\ast}g_{k}^{(m)}\rightharpoonup 0
	\mbox{ for }n\neq m.&\\ 
	\qquad \qquad \qquad \qquad \qquad &\sum_{n\in\mathbb{N}_0}\|w^{(n)}\|^{2}\leq\delta^{-1}\limsup _k\|u_k\|^2.&\\ 
	\qquad \qquad \qquad \qquad \qquad &u_k-\sum_{n\in\mathbb{N}_0}g_{k}^{(n)}w^{(n)}\stackrel{D}{\rightharpoonup}0, &
	\end{flalign*}
	where the series $\sum_{n\in\mathbb{N}_0}g_{k}^{(n)}w^{(n)}$ converges uniformly in $k.$
\end{theoremletter}
\section{Proof of Theorem~\ref{teo_tinta_frac}}\label{assumptions}
We follow the same arguments used in \cite[Section 5]{palatucci}. First, we establish some notation.

Given $\gamma >1, $ let
\begin{equation}\label{dilat_group}
\delta_{\mathbb{R}} := \left\lbrace \delta _j : \mathcal{D}^{s,2} (\mathbb{R}^N) \rightarrow \mathcal{D}^{s,2} (\mathbb{R}^N) : \delta _j u(x) = \gamma ^{\frac{N-2s}{2} j} u (\gamma ^j x ),\ j \in \mathbb{R} \right\rbrace,
\end{equation}
and
\begin{equation*}
T_{\mathbb{R}^N } := \left\lbrace g_y : \mathcal{D}^{s,2} (\mathbb{R}^N) \rightarrow \mathcal{D}^{s,2} (\mathbb{R}^N) : g_y u(x) = u(x-y), \ y \in \mathbb{R}^N \right\rbrace,
\end{equation*}
the groups of operators on $\mathcal{D}^{s,2} (\mathbb{R}^N)$ induced by dilations and translations on $\mathbb{R}^N,$ respectively. It is easy to see that $T_{\mathbb{R}^N } $ and $\delta_{\mathbb{R}}$ are groups of unitary operators in $\mathcal{D}^{s,2} (\mathbb{R}^N),$ by using the following identities
\begin{equation}\label{frac_id}
\left\{
\begin{aligned}
& (-\Delta )^{s/2} \left( u (\cdot - y) \right) = \left( (-\Delta )^{s/2} u \right)(\cdot - y),&\\
& (-\Delta )^{s/2}\left( u(\tau \cdot) \right)=\tau ^{s} \left( (-\Delta )^{s/2}u \right) (\tau \cdot),&
\end{aligned}
\right.
\end{equation}
$u \in \mathcal{D}^{s,2} (\mathbb{R}^N),\ y \in \mathbb{R}^N$ and $\tau>0.$ Thus $D_{\mathbb{R}^N }$ (defined in \eqref{D_erren}) consists of compositions of the elements of $T_{\mathbb{R}^N }$ with $\delta_{\mathbb{R}},$ i.e., $d_{y,j} = \delta _j \circ g_{\gamma ^j y }.$ By checking that $d_{y,j} \circ d_{z,l} = d_{y+ \gamma ^{-j}z , j + l}$ and $(d_{y,j})^{-1} = d _{- \gamma ^j y, -j},$ we see that $D_{\mathbb{R} ^N}$ is a group of unitary operators in $\mathcal{D}^{s,2} (\mathbb{R}^N).$

The following results describe how the elements of $D_{\mathbb{R} ^N}$ acts in  $\mathcal{D}^{s,2} (\mathbb{R}^N).$ They are similar to the ones in \cite[Section 5]{palatucci} and can be proved using analogous arguments.
\begin{lemma}\label{op_conv}
Let $(y_k, j_k) \subset \mathbb{R}^N \times \mathbb{R},$ such that $(y_k, j_k) \rightarrow (y, j).$ Then $d_{y_k, j_k} u \rightarrow d_{y, j} u,$ for all $u \in \mathcal{D}^{s,2} (\mathbb{R}^N).$
\end{lemma}
\begin{lemma}\label{lemma5.1}
Let $u \in \mathcal{D}^{s,2} (\mathbb{R}^N) \setminus \{ 0 \}.$ The sequence $(d_{y_k,j_k}u ),$ with $(y_k,j_k) \subset \mathbb{R}^N \times \mathbb{R},$ converges weakly to zero if and only if $|j_k| + |y_k| \rightarrow \infty.$
\end{lemma}
The idea to prove Theorem \ref{teo_tinta_frac} is to take $D=D_{\mathbb{Z} ^N}$ and apply Theorem \ref{teo_tinta}. The main reason to take  $D=D_{\mathbb{Z} ^N} $ (instead of $D_{\mathbb{R} ^N}$) is described in Sect. \ref{Weak Convergence}: it gives further properties for the weak decomposition. 

Considering the next cocompactness result we are able to prove Theorem \ref{teo_tinta_frac}. As a consequence of it and \cite[Proposition 1]{palatucci}, we see that $D_{\mathbb{R}^N}$--weak convergence is equivalent to the $D_{\mathbb{Z}^N}$--weak convergence in $\mathcal{D}^{s,2}(\mathbb{R}^N),$ for $0<s<1.$
\begin{proposition}\label{cocompact}
Assume that $0<s<\min\{1,N/2\}.$ Let $(u_k)$ be a bounded sequence in $\mathcal{D}^{s,2}(\mathbb{R} ^N).$ Then $u_k \stackrel{D}{\rightharpoonup} 0$ if and only if $u_k \rightarrow 0$ in $L^{2 ^\ast_s}(\mathbb{R}^N).$
\end{proposition}
\begin{proof}
Our proof follows the same ideas of \cite[Lemma 5.3]{tintabook}.
Since $C^\infty_0 (\mathbb{R}^N)$ is a dense subset of $\mathcal{D}^{s,2}(\mathbb{R}^N),$ by the continuous embedding of $\mathcal{D}^{s,2}(\mathbb{R}^N)$ in $L^{2_s ^\ast }(\mathbb{R}^N),$ we can assume without loss of generality that the sequence $(u_k)$ belongs to $C^\infty_0 (\mathbb{R}^N).$ Let us suppose first that $u_k \stackrel{D}{\rightharpoonup} 0$. Consider $\xi \in C^\infty _0 (\mathbb{R}, [0,\infty))$ such that
\begin{equation*}
\xi(t)=
\left\{
\begin{aligned}
t,&\quad\text{if }\frac{1}{4} \gamma^{\frac{N-2s}{2}}  \leq t\leq \frac{3}{4}\gamma^{\frac{N-2s}{2}},\\
0,&\quad \text{if }t\leq 1\text{ or }t \geq \gamma^{\frac{N-2s}{2}},
\end{aligned}
\quad\text{and}\quad |\xi'(t)| \leq C,\ \forall t,
\right.
\end{equation*}
where we can assume without loss of generality that $\gamma > 4,$ because we can replace it by $\gamma ^{n_0} >4,$ for integer $n_0$ large enough, if necessary. Notice that there exists a positive constant $C$ such that
\begin{equation}\label{cutoff}
\left\{
\begin{aligned}
& |\xi(t)|^{2_s ^\ast} \leq C t^2,\\
&|\xi(t)|^{2} \leq C |t|^{2_s ^\ast},
\end{aligned}
\right.
\quad \forall t.
\end{equation}
Given any sequence $(j_k)$ in $\mathbb{Z},$ denote
\begin{equation*}
v_k (x) = \gamma^{\frac{N-2s}{2}j_k} u_k (\gamma^{j_k} x).
\end{equation*}
Let $Q_z = (0,1)^N + z,$ with $z \in \mathbb{Z}^N.$ By the Sobolev embedding \eqref{comp_loc}, for any $z \in \mathbb{Z}^N,$ we get that
\begin{equation}\label{expressao}
\int_{Q_z} |\xi (|v_k|)|^{2_s ^\ast} \dx \leq C \|\xi (|v_k|) \|^2_{H^s (Q_z)} \left(\int_{Q_z} v^2_k \dx \right)^{1-2/2_s^\ast}.
\end{equation}
Moreover, embedding \eqref{Hembd} and relations \eqref{cutoff} implies that,
\begin{align*}
\sum _{z \in \mathbb{Z}} \|\xi (|v_k|) \|^2_{H^s (Q_z)} &= \sum _{z \in \mathbb{Z}} \int _{Q_z } |\xi (|v_k|)|^2\dx + \int _{Q_z} \int _{Q_z} \frac{\left| \xi (|v_k|)(x) - \xi (|v_k|)(y) \right|^2}{|x-y|^{N + 2s}}\dxdy \\
&\leq \int _{\mathbb{R}^N} |\xi (|v_k|)|^2\dx + \max_{t\geq 0} \xi' (t) \sum _{z \in \mathbb{Z}} \int _{Q_z} \int _{Q_z} \frac{\left| v_k(x) -  v_k(y) \right|^2}{|x-y|^{N + 2s}}\dxdy\leq C \|v_k\|^2.
\end{align*}
Thus, we can take the sum over $z \in \mathbb{Z}^N$ in \eqref{expressao} to obtain
\begin{equation}\label{eROU1}
\int_{\mathbb{R}^N} |\xi (|v_k|)|^{2_s ^\ast} \dx \leq C \sup_{z \in \mathbb{Z}^N} \left(\int_{Q_z} v^2_k \dx \right)^{1-2/2_s^\ast}.
\end{equation}
For each $k,$ let $z_k \in \mathbb{Z}^N$ such that
\begin{equation}\label{aux_cc_1}
\sup_{z \in \mathbb{Z}^N} \left(\int_{Q_z} v^2_k \dx \right)^{1-2/2_s^\ast} \leq 2 \left(\int_{Q_{z_k}} v^2_k \dx \right)^{1-2/2_s^\ast}.
\end{equation}
Since $u_k \stackrel{D}{\rightharpoonup} 0,$ we have that $v_k (\cdot - z_k) \rightharpoonup 0$ in $\mathcal{D}^{s,2}(\mathbb{R}^N),$ which allow us to apply embedding \eqref{comp_loc} and obtain that
\begin{equation}\label{aux_cc_2}
\int_{Q_{z_k}} v^2_k \dx = \int_{(0,1)^N} v^2_k (\cdot - z_k)\dx \rightarrow 0,\text{ as }k \rightarrow \infty.
\end{equation}
Replacing \eqref{aux_cc_1} and \eqref{aux_cc_2} in \eqref{eROU1} we conclude that
\begin{equation}\label{concluc1}
\lim _{k \rightarrow \infty} \int_{\mathbb{R}^N} |\xi (|v_k|)|^{2_s ^\ast} \dx = 0.
\end{equation}
Now let
\begin{equation*}
\xi _j (t) = \gamma ^{- \frac{N-2s}{2}j} \xi (\gamma ^{\frac{N-2s}{2}j} t),\quad j \in \mathbb{Z}.
\end{equation*}
From convergence \eqref{concluc1}, we get
\begin{equation}\label{conclusao_cc}
\lim _{k \rightarrow \infty} \int_{\mathbb{R}^N} |\xi _{j_k}(|u_k|) | ^{2_s ^\ast} \dx=\lim _{k \rightarrow \infty} \int_{\mathbb{R}^N} |\xi (|v_k|)|^{2_s ^\ast} \dx = 0,\quad\text{for any sequence }(j_k)\text{ in }\mathbb{Z}. 
\end{equation}
Now the embedding $\mathcal{D}^{s,2}(\mathbb{R}^N) \hookrightarrow L ^{2^{\ast } _s} (\mathbb{R}^N)$ enable us to get the following estimate,
\begin{equation}\label{expressao2}
\int_{\mathbb{R}^N} |\xi _{j}(|u_k|) | ^{2_s ^\ast} \dx \leq C \|\xi _{j}(|u_k|) \|^2 \left(\int_{\mathbb{R}^N} |\xi _{j}(|u_k|) | ^{2_s ^\ast} \dx \right)^{1-2/2_s ^\ast }.
\end{equation}
For $j \in \mathbb{Z},$ let 
\begin{equation*}
\left\{
\begin{aligned}
D_{j,k} &= \left\lbrace x \in \mathbb{R}^N :  \gamma^{-\frac{N-2s}{2}j} \leq |u_k(x)| < \gamma ^{-\frac{N-2s}{2}(j-1)}\right\rbrace ,\\
E_{j,k} &= (D_{j,k} \times \mathbb{R}^N ) \cup (\mathbb{R}^N \times D_{j,k}),\\
L_{j,k} &= \left\lbrace x\in \mathbb{R}^N :  \frac{1}{4} \gamma^{-\frac{N-2s}{2}  j}  \leq |u_k(x)|\leq \frac{3}{4}\gamma^{-\frac{N-2s}{2}(j-1)}\right\rbrace.
\end{aligned}
\right.
\end{equation*}
Since $u_k$ is smooth and has compact support, there exists $j_0$ in $\mathbb{Z}$ and $l$ in $\mathbb{N}$ such that
\begin{equation*}
\supp(u_k) \subset \bigcup_{j=0}^l L_{ j+j_0,k} \subset \bigcup _{j=0} ^l D_{ j+j_0,k},
\end{equation*}
We also have that the sets
\begin{equation*}
S_{j,k}=\bigcup _{m=0}^j  E_{j+j_0,k} \cap E_{m+j_0,k},\quad j=1,\ldots,l,
\end{equation*}
are disjunct as well $E_{j_0,k}$ and $E_{j+j_0,k}\setminus S_{j,k},$ for $j=1,\ldots,l.$ Thus we may write
\begin{align*}
\sum _{j =0 }^l \iint _{E_{j+j_0,k}} \frac{\left| u_k(x) -  u_k(y) \right|^2}{|x-y|^{N + 2s}} &\dxdy = \sum_{j=1}^l \iint _{S_{j,k}} \frac{\left| u_k(x) -  u_k(y) \right|^2}{|x-y|^{N + 2s}}\dxdy \\+\iint _{E_{j_0,k}} & \frac{\left| u_k(x) -  u_k(y) \right|^2}{|x-y|^{N + 2s}} \dxdy+ \sum_{j=1}^l \iint _{E_{j+j_0,k}\setminus S_{j,k}} \frac{\left| u_k(x) -  u_k(y) \right|^2}{|x-y|^{N + 2s}}\dxdy,&\\
&=\iint _{A_{l,k}} \frac{\left| u_k(x) -  u_k(y) \right|^2}{|x-y|^{N + 2s}} \dxdy + \iint _{B_{l,k}} \frac{\left| u_k(x) -  u_k(y) \right|^2}{|x-y|^{N + 2s}} \dxdy,
\end{align*}
where
\begin{equation*}
A_{l,k} = E_{j_0,k}\cup \bigcup _{j=1} ^l E_{j+j_0,k} \setminus S_{j,k}\quad\text{and}\quad B_{l,k} = \bigcup _{j=1} ^l S_{j,k},
\end{equation*}
to get that the estimate
\begin{align*}
\sum _{j =0 }^l\|\xi _{j}(|u_k|) \|^2 &=\frac{C(N,s)}{2} \sum _{j =0 }^l \iint _{E_{j,k}} \frac{\left| \xi _{j}(|u_k|)(x) -  \xi _{j}(|u_k|)(y) \right|^2}{|x-y|^{N + 2s}}\dxdy\\
&\leq\frac{C(N,s)}{2} \max_{t\geq 0} \xi' (t) \sum _{j =0 }^l \iint _{E_{j,k}} \frac{\left| u_k(x) -  u_k(y) \right|^2}{|x-y|^{N + 2s}}\dxdy\leq 2 \max_{t\geq 0} \xi' (t)  \| u_k \|^2.
\end{align*}
Moreover,
\begin{multline*}
\int_{\mathbb{R}^N} |u_k|^{2_s^\ast}\dx \leq \sum _{j =0 }^l \int_{ L_{j,k}} |u_k|^{2_s ^\ast}\dx \\ \leq \sum _{j =0 }^l \int_{ L_{j,k}}|u_k|^{2_s ^\ast}\dx + \int_{D_{j,k} \setminus L_{j,k}} |\xi_j (|u_k|)|^{2_s ^\ast } \dx =\sum _{j =0 }^l \int_{\mathbb{R}^N}|\xi _j(|u_k|)|^{2_s ^\ast } \dx.
\end{multline*}
In view of that, we take the sum over $j=0,\ldots,l$ in \eqref{expressao2} to conclude that
\begin{equation*}
\int_{\mathbb{R}^N} |u_k|^{2_s^\ast}\dx \leq C \sup_{j \in \mathbb{Z}} \left(\int_{\mathbb{R}^N} |\xi _{j}(|u_k|) | ^{2_s ^\ast} \dx \right)^{1-2/2_s ^\ast }.
\end{equation*}
Similarly as before, we choose $(j_k)$ such that
\begin{equation*}
\sup_{j \in \mathbb{Z}} \left(\int_{\mathbb{R}^N} |\xi _{j}(|u_k|) | ^{2_s ^\ast} \dx \right)^{1-2/2_s ^\ast } \leq 2 \left(\int_{\mathbb{R}^N} |\xi _{j_k}(|u_k|) | ^{2_s ^\ast} \dx \right)^{1-2/2_s ^\ast },
\end{equation*}
which, from \eqref{conclusao_cc} implies that $|u_k|_{2^\ast_s} \rightarrow 0.$

Now assume that $u_k \rightarrow 0$ in $L^{2 ^\ast_s}(\mathbb{R}^N).$ Let us argue by contradiction and suppose that there exists $(y_k)$ in $\mathbb{Z}^N$ and $(j_k)$ in $\mathbb{Z}$ such that $d_{y_k,j_k} u_k \rightharpoonup u \neq 0$ in $\mathcal{D}^{s,2}(\mathbb{R}^N).$ The invariance of $d_{y_k,j_k}$ with respect to the $L^{2_s ^\ast}$ norm leads to
\begin{equation*}
| u |_{2_s^\ast} \leq \liminf _k| d_{y_k,j_k} u_k |_{2_s^\ast} = \lim_{k \rightarrow \infty } |u_k |_{2_s^\ast} = 0,
\end{equation*} 
which is a contradiction with the fact that $u \neq 0.$
\end{proof}
\begin{proof}[Proof of Theorem \ref{teo_tinta_frac} completed]
By Theorem \ref{teo_tinta}, we first need to prove that $(\mathcal{D}^{s,2} (\mathbb{R}^N), D_{\mathbb{Z} ^N , \mathbb{Z}})$ is a dislocation space. To do so, we use Proposition \ref{prop3.1}. Let $(d_{y_k,j_k} ) \subset D_{\mathbb{Z} ^N , \mathbb{Z}},$ such that  $d_{y_k,j_k} \not \rightharpoonup 0$ in $\mathcal{D}^{s,2} (\mathbb{R}^N).$ Hence by Lemma \ref{lemma5.1}, $y_k \rightarrow y$ and $j_k \rightarrow j,$ up to a subsequence, and by Lemma \ref{op_conv}, $d_{y_k,j_k} u \rightarrow d_{y,j}u,$ for all $u \in \mathcal{D}^{s,2} (\mathbb{R}^N).$ Therefore Theorem \ref{teo_tinta} holds with $H=\mathcal{D}^{s,2} (\mathbb{R}^N)$ and $D=D_{\mathbb{Z} ^N , \mathbb{Z}}$. It follows immediately assertions \eqref{seis.um} and \eqref{seis.tres}. The assertion \eqref{seis.dois} is guaranteed by Lemma \ref{lemma5.1}, and \eqref{seis.tres} follows from Proposition \ref{cocompact}. Finally, for each $n \in \mathbb{N}_\ast,$ if $(j_k ^{(n)})$ is unbounded we can replace it by a subsequence convergent to $+ \infty$ or $=-\infty,$ by checking either $\limsup_k j_k ^{(n)} = + \infty$ or $\limsup _k j_k ^{(n)} = - \infty.$ If $(j_k ^{(n)}) $ is bounded, we can replace it by a constant subsequence, say $j ^{(n)}.$ Moreover, by taking $v_k ^{(n)} (x)= \gamma ^{-\frac{N-2s}{2}j ^{(n)}}u_k (\gamma ^{- j {(n)}} x + y_k^{(n)} ),$ the convergence \eqref{seis.um} implies
\begin{equation*}
u_k (\cdot + y_k ^{(n)}) = \delta _{- j ^{(n)} }v_k ^{(n)} \rightharpoonup \delta _{- j ^{(n)} } w^{(n)} \text{ in }\mathcal{D}^{s,2} (\mathbb{R}^N),
\end{equation*}
thus we may set $j^{(n)} = 0$ and rename $\delta _{- j ^{(n)} } w^{(n)}$ as $ w^{(n)}.$ Since $\mathbb{N}_\ast$ is possibly infinite, the conclusion follows by a standard diagonal argument in the extraction of each successive subsequence. \end{proof}

\section{Self-similar functions}\label{self-similar_functions}
We now pass to study a class of non-linearity consistent with our profile decomposition. As it can be seen in the following examples, this class of nonlinearity can been seen as asymptotically oscillatory about the critical power $|t|^{2^\ast _s}$ and not satisfying that $t^{-1}f(x,t)$ is an increasing function.
\begin{definition}\label{defi_ss}
We say that $F \in C(\mathbb{R})$ is fractional self-similar if there exist $\gamma >1$ and $0<s<\min\{1,N/2\}$ such that
 \begin{equation*}
 F(t)=\gamma ^{-Nj} F(\gamma ^{\frac{N-2s}{2}j} t ), \;  \forall j \in \mathbb{Z}, \; t \in \mathbb{R}.
 \end{equation*}
In this case we use to say that $F$ is fractional self-similar 
 with factor $\gamma$ and fraction $s$.
 \end{definition}
\begin{example}
Typical examples of self-similar functions are
\begin{enumerate}[label=(\roman*)]
\item $F(t)=|t|^{2^\ast _s},$ which is self-similar for every factor $\gamma$ and fraction $0<s<\min\{1,N/2\};$
\item $H(t)=\cos (\ln |t| ) |t|^{2^\ast _s},$ $H(0):=0,$ which is self-similar with factor $e^{4 \pi /(N-2s)}$ and every fraction $0<s<\min\{1,N/2\}.$
\end{enumerate}
\end{example}
\begin{remark}
The function $F(t) \in C^1 (\mathbb{R})$ is self-similar if, and only if
\begin{equation*}
F'(t)=\gamma ^{- \frac{N+2s}{2}j} F'\left(\gamma ^{ \frac{N-2s}{2}j}  t \right), \; \forall j \in \mathbb{Z}, \text{ and }t \in \mathbb{R}.
\end{equation*}
Consequently, we can say that $f(t)\in C(\mathbb{R})$ is self-similar whenever its primitive $F(t) = \int _0 ^t f(\tau) d\tau$ satisfies the condition of Definition \ref{defi_ss}. For the local case a class of self-similar function was introduced in \cite{tinta_pos,tintabook,tintapaper}.
\end{remark}
In the next result we derive the basic properties of self-similar functions.
\begin{lemma}\label{basic_prop_ss}
Assume that $F(t)$ is self-similar.
\begin{enumerate}[label=(\roman*)]
\item\label{bpss1} For each $u \in L^{2_s ^{\ast }} (\mathbb{R}^N)$ and $j \in \mathbb{Z},$ we have
\begin{equation}\label{absor}
\int _{\mathbb{R}^N} F \left(\gamma ^{\frac{N-2s}{2}j} u(\gamma ^j x )\right) \dx = \int _{\mathbb{R}^N} F (u) \dx;
\end{equation}
\item\label{bpss2} There exists $C>0$ such that
\begin{equation}\label{CRITG}
	|F(t)|\leq C |t|^{2_s ^\ast}, \; \forall t \in \mathbb{R}.
\end{equation}
Moreover, if $F \in C^2 (\mathbb{R}),$ then there exists $C>0,$ such that
\begin{equation}\label{assumption_reg}
|F(t)|+|F'(t)t|+|F''(t) t^2|\leq C |t|^{2_s ^\ast}, \; \forall  t \in \mathbb{R};
\end{equation}
\item\label{bpss3} If $F(t)$ is locally Lipschitz then $F(t)$ satisfies \eqref{f_extra}.
\end{enumerate}
\end{lemma}
\begin{proof}
\ref{bpss1} The identity \eqref{absor} follows immediately by using the change of variables theorem in the integral on the left side of the equation.\\
\ref{bpss2} Fix the interval $L=[\gamma ^{-\frac{N-2s}{2}}, \gamma ^{\frac{N-2s}{2}}].$ By continuity, there exists $C=C(L)$ such that $|F(t)| \leq C t^{2_s ^\ast},$ for all $t \in L.$ Now, let $0<t<\gamma ^{-\frac{N-2s}{2}}$ or $t>\gamma ^{\frac{N-2s}{2}},$ then (in any case) there exists $j \in \mathbb{Z}$ such that $\gamma ^{\frac{N-2s}{2} j}t \in L,$ and consequently, 
\begin{equation*}
\gamma ^{Nj} |F(t)| = |F(\gamma ^{\frac{N-2s}{2}j}t)|\leq \gamma ^{Nj} C t^{2^\ast _s}.
\end{equation*}
The case where $t<0$ is analogous. The proof of \eqref{assumption_reg} follow a similar argument.\\
\ref{bpss3} The proof is by induction in $M.$ So we just need to prove that there exits $C>0$ such that
\begin{equation}\label{case_um}
\left|F(a_1 + a_2) - F(a_1) - F(a_2)\right| \leq C \left( |a_1| |a_2| ^{2^\ast _s -1} + |a_1| ^{2^\ast _s -1} |a_2| \right).
\end{equation}
To do so, we first fix the interval $I = [-\gamma ^{\frac{N-2s}{2} k}, \gamma ^{\frac{N-2s}{2} k}],$ where $k \in \mathbb{Z}$ is taken such that $\gamma ^{\frac{N-2s}{2} (k-1)} > 2,$ to use the Lipschitz assumption. The proof follows by considering several cases.\\
\textit{Case 1:} Suppose that $|a_1| \leq 1 \leq |a_2|$ and $a_1 + a_2 \in I.$ Thus there exists $C=C(I)$ such that
\begin{equation*}
\left|F(a_1 + a_2) - F(a_1) - F(a_2)\right| \leq C (|a_1| + |F(a_1)| ).
\end{equation*}
By condition \eqref{CRITG} we can estimate
\begin{equation*}
|a_1| + |F(a_1)| \leq C(|a_1| |a_2| ^{2^\ast _s -1} + |a_1| ^{2^\ast _s -1} |a_2| ).
\end{equation*}
\textit{Case 2:} Assume that $|a_1| , |a_2| \geq 1$ and $a_1 + a_2 \in I.$ Then, there exists $j_1 \in \mathbb{Z},$ $j_1 \leq 0,$ such that $|b_1| \leq 1,$ where $b_1 := \gamma ^{\frac{N-2s}{2}j_1} a_1.$ It is easy to see that $b_1 + a_2 \in I,$ hence by the first case, we have the following estimate
\begin{align*}
|F(b_1 + a_2) - F(b_1) - F(a_2)| &\leq \gamma ^{\frac{N-2s}{2}j_1} C(|a_1| ^{2^\ast _s -1} |a_2| + |a_1| |a_2| ^{2^\ast _s -1} )\\
&\leq C(|a_1| ^{2^\ast _s -1} |a_2| + |a_1| |a_2| ^{2^\ast _s -1} ),
\end{align*}
Therefore we can estimate as follows
\begin{multline*}
\left|F(a_1 + a_2) - F(a_1) - F(a_2)\right| \leq \\ \left|F(b_1 + a_2) - F(b_1) - F(a_2)\right| + \left|F(a_1 + a_2) - F(b_1 + a_2) + F(b_1) - F(a_1) \right|,
\end{multline*}
with
\begin{equation*}
|F(a_1 + a_2) - F(a_1) - F(b_1 + a_2) + F(b_1)| \leq 2C |a_2| \leq C |a_1| ^{2_s ^\ast } |a_2|.
\end{equation*}
\textit{Case 3:} Suppose that $|a_1|,|a_2| \leq 1.$ Since
\begin{equation*}
\mathbb{R} =\bigcup _{j \in \mathbb{Z}} I^-_j  \cup I^+_j,
\end{equation*}
where $I ^{-}_j = [-\gamma ^{\frac{N-2s}{2}j} , -\gamma ^{\frac{N-2s}{2}(j-1)}]$ and $I ^{+}_j = [\gamma ^{\frac{N-2s}{2}(j-1)} , \gamma ^{\frac{N-2s}{2}j}]$ there exists $j_0 \in \mathbb{Z}$ such that
\begin{equation*}
\gamma ^{\frac{N-2s}{2}j_0} (a_1 + a_2) \in \left[- \gamma ^{\frac{N-2s}{2} k}, - \gamma ^{\frac{N-2s}{2}(k-1)} \right] \cup \left[\gamma ^{\frac{N-2s}{2} (k-1)}, \gamma ^{\frac{N-2s}{2} k} \right]
\end{equation*}
Let $b_1  = \gamma ^{\frac{N-2s}{2}j_0} a_1$ and $b_2 = \gamma ^{\frac{N-2s}{2}j_0} a_2,$ with the necessity $|b_1| \geq 1 $ or $|b_2| \geq 1,$ because $\gamma ^{\frac{N-2s}{2}(k-1)} > 2.$ Consequently we can use the first or the second case to get that
\begin{align*}
\gamma ^{Nj_0} |F(a_1 + a_2) - F(a_1) - F(a_2)| &= |F(b_1+b_2) - F(b_1) - F(b_2)|\\
 & \leq \gamma ^{Nj_0} C (|a_1| ^{2^\ast _s -1} |a_2| + |a_1| |a_2| ^{2^\ast _s -1} ).
\end{align*}
The general case follows by a similar argument as above, thus we conclude that \eqref{case_um} holds.
\end{proof}
\begin{remark}\label{they_are_ss}
If $f(x,t)$ satisfies \eqref{f_5} then
\[
\left\{	
	\begin{aligned}
	F_0 (t) &= \lim _{|x| \rightarrow \infty} F(x,t). \\
	F_+ (t) &= \lim _{j \in \mathbb{Z}, j \rightarrow + \infty} \gamma ^{-Nj} F\left(\gamma ^{-j}x , \gamma ^{\frac{N-2s}{2}j}t\right),\\
	F_- (t) &= \lim _{j \in \mathbb{Z}, j \rightarrow - \infty} \gamma ^{-Nj} F\left(\gamma ^{-j}x , \gamma ^{\frac{N-2s}{2}j}t\right).
	\end{aligned}
	\right.
	\]
uniformly in compact sets. Furthermore, $F_+(t)$ and $F_{-}(t)$ are self-similar.
\end{remark}
\section{On the behavior of weak decomposition convergence under nonlinearities}\label{onbehaviour}
Concerning the assumptions \eqref{f_extra}, \eqref{f_5}, and \eqref{selfsimilar}, we have the following results, which provides a way to link the weak convergence decomposition (as also the latter lines of Theorem \ref{teo_tinta_frac}) and the limit over the energy functional $I$ for bounded sequences in $\mathcal{D}^{s,2} (\mathbb{R}^N).$ They are mainly used to prove the existence results stated in Sect. \ref{Application}. Also, the next result can be seen as a generalization of the well know Brezis-Lieb Lemma \cite{brezis-lieb} (see Corollary \ref{brezis-lieb}).
\begin{proposition}\label{lemma2}
	Let $0<s<\min\{1,N/2\}$ and assume that $f(x,t)$ satisfies \eqref{f_1}, \eqref{f_2}, \eqref{f_extra} and \eqref{f_5}. Let $(u_k)$ in $\mathcal{D}^{s,2} (\mathbb{R}^N)$ be a bounded sequence and $(w ^{(n)}) _{n \in \mathbb{N}_{\ast } }$ in $\mathcal{D}^{s,2} (\mathbb{R}^N),$ $n \in \mathbb{N}_{\ast },$ provided by Theorem \ref{teo_tinta_frac}. Then
	\begin{multline}\label{eqlemma2}
	\lim _{k \rightarrow \infty} \int _{\mathbb{R}^N} F(x,u_k) \dx = \int _{\mathbb{R}^N} F(x,w ^{(1)}) \dx\\+\sum_{n \in \mathbb{N} _0,n>1} \int _{\mathbb{R}^N} F_0 (w ^{(n)}) \dx +  \sum_{n \in \mathbb{N} _{+ }} \int _{\mathbb{R}^N} F_+ (w^{(n)}) \dx + \sum_{n \in \mathbb{N} _{- }} \int _{\mathbb{R}^N} F_{-} (w^{(n)}) \dx.
	\end{multline}
\end{proposition}
\begin{proof}
	Let us first introduce the notation
	\begin{equation*}
	d _k ^{(n)} u (x)= \gamma ^{\frac{N-2s}{2} j_k ^{(n)} }u(\gamma^{j _k ^{(n)}} ( x - y_k^{(n)} ) ),\quad u \in \mathcal{D}^{s,2} (\mathbb{R}^N),
	\end{equation*}
	where $(y_k ^{(n)}) _ {k \in \mathbb{Z}} \subset \mathbb{Z}^N,$ $(j _k ^{(n)}) _{k \in \mathbb{N}}  \subset \mathbb{Z}.$ By \eqref{f_2}, the functional 
	\begin{equation*}
	\Phi(u)=\int_{\mathbb{R}^N} F (x,u) \dx,\quad u \in \mathcal{D}^{s,2}(\mathbb{R}^N),
	\end{equation*}
is uniformly continuous in bounded sets of $L^{2_s ^{\ast }} (\mathbb{R}^N),$ which implies (by \eqref{seis.tres} and \eqref{seis.quatro}) that
	\begin{equation*}
	\lim _{k \rightarrow \infty}\left[\Phi (u_k) - \Phi \left( \sum _{n \in \mathbb{N}_{\ast }} d_k ^{(n)} w^{(n)}\right) \right] =0.
	\end{equation*}
	To prove \eqref{eqlemma2} we observe that the uniform convergence of the series in \eqref{seis.quatro} allows us to consider only the case where $\mathbb{N}_\ast= \{1,\ldots, M\}.$ Thus,
	\begin{align}
	&\lim _{k \rightarrow \infty}\left[\sum _{n \in \mathbb{N}_{0 }} \Phi \left(  w^{(n)}( \cdot - y_k^{(n)} ) \right) -\Phi(w^{(1)})-\sum_{n \in \mathbb{N} _0,n>1} \Phi_0 (w ^{(n)})\right] = 0\label{giant2},\\
	&\lim _{k \rightarrow \infty}\left[ \sum _{n \in \mathbb{N}_{\pm  }}  \Phi( d_k ^{(n)} w^{(n)}) - \sum _{n \in \mathbb{N}_{\pm  }} \Phi_{\pm } (w ^{(n)})  \right]=0, \label{giant3}
	\end{align}
	follows immediately from the assumption \eqref{f_5}, by change of variables and the use of Lebesgue Convergence Theorem. Therefore it is sufficient to prove that
	\begin{equation}\label{giant1}
	\lim _{k\rightarrow\infty}\left[ \Phi \left( \sum _{n \in \mathbb{N}_{\ast }} d_k ^{(n)} w^{(n)}\right)-  \sum _{n \in \mathbb{N}_{\ast }} \Phi ( d_k ^{(n)} w^{(n)} ) \right]=0.
	\end{equation}
	Indeed, by \eqref{f_extra} we have for all $m \neq n,$
	\begin{equation*}
	\left| \Phi \left( \sum _{n \in \mathbb{N}_{\ast }} d _k ^{(n)} w ^{(n)}\right)-  \sum _{n \in \mathbb{N}_{\ast }} \Phi (d _k ^{(n)} w ^{(n)} ) \right| \leq \sum _{m \neq n \in \mathbb{N} _{\ast }} \int _{\mathbb{R}^N} |d _k ^{(n)} |^{2 _s ^\ast -1}  | d _k ^{(m)} | \dx.
	\end{equation*}
	But by a change of variable we can see that
	\begin{equation*}
	\int _{\mathbb{R}^N} | d _k ^{(n)} |^{2 _s ^\ast -1}  | d _k ^{(m)} | \dx = \int _{\mathbb{R}^N} |w^{(n)}| ^{2 _s ^\ast -1} g_k(|w^{(m)}|) \dx,
	\end{equation*}
	where 
	\begin{equation*}
	g_k(|w^{(m)}|) = \gamma ^{\frac{N-2s}{2} (j _k ^{(m)} - j _k ^{(n)} )}  w^{(m)} \left( \gamma ^{j _k ^{(m)} - j _k ^{(n)} } (\cdot-\gamma ^{j_k ^{(n)}} (y_k ^{(m)} - y _k ^{(n)})  )\right) \rightharpoonup 0\text{ in } \mathcal{D}^{s,2} (\mathbb{R}^N),
	\end{equation*}
due to \eqref{seis.dois} and Lemma \ref{lemma5.1}. Since 
\begin{equation*}
\alpha(v) = \int _{\mathbb{R}^N} |w^{(n)}| ^{2 _s ^\ast -1} v \dx
\end{equation*}
is a continuous linear functional in $\mathcal{D}^{s,2} (\mathbb{R}^N)$ we conclude \eqref{giant1}.
\end{proof}
\begin{corollary}\label{rightone}
Let $(u_k)$ be a bounded sequence in $\mathcal{D}^{s,2} (\mathbb{R}^N)$ and $(w ^{(n)}) _{n \in \mathbb{N}_{\ast } }$ in $\mathcal{D}^{s,2} (\mathbb{R}^N),$ $n \in \mathbb{N}_{\ast },$ provided by Theorem \ref{teo_tinta_frac}. If $F(x,t)=F(t)$ satisfies \eqref{selfsimilar} and is locally Lipschitz then, on up to a subsequence,
\begin{equation}\label{lemaUM}
\lim _{k \rightarrow \infty} \int _{\mathbb{R}^N } F(u_k) \dx = \sum _{n \in \mathbb{N} _{\ast }} \int _{\mathbb{R}^N} F(w^{(n)}) \dx.
\end{equation}
\end{corollary}
\begin{proof}
In this case $F(t)$ satisfies \eqref{f_extra} and \eqref{CRITG}, also $F=F_+ = F_{-}=F_0$.
\end{proof}
\begin{corollary}\label{brezis-lieb}
Let $u_k \rightharpoonup u$ in $\mathcal{D}^{s,2} (\mathbb{R}^N)$ and $F(t)$  be as in Corollary \ref{rightone} then, up to a subsequence,
\begin{equation*}
\lim _{k\rightarrow \infty}\int _{\mathbb{R}^N} F(u_k) - F(u - u_k) - F(u) \dx = 0.
\end{equation*}
\end{corollary}
\begin{proof}
Since $w^{(1)} = u,$ by \eqref{seis.quatro} and Corollary \ref{rightone} we have
\begin{equation}\label{lemaDOIS}
\lim _{k\rightarrow \infty} \int _{\mathbb{R}^N} F(u_k - u) \dx = \sum _{n\in \mathbb{N}_\ast,n>1} \int _{\mathbb{R}^N} F(w^{(n)}) \dx.
\end{equation}
Taking the difference between \eqref{lemaUM} and \eqref{lemaDOIS} we get the desired result.
\end{proof}
\section{The autonomous case}\label{asfec}
The aim of this section is to prove Theorems \ref{primeira}, \ref{segunda} and \ref{minimizacao}.
\begin{remark}\label{1rmax+-}
	By embedding \eqref{embds}, we have $\mathcal{S}_l < \infty.$ Also $\mathcal{S} _l$ is attained for some $l$ if and only if it is attained for all $l.$ Indeed, this can be checked by considering the rescaling $v(x)=u_1 (l^{-1/(N-2s)} x)$ and $u(x)=v_l (l^{1/(N-2s)} x),$ where $\|u_1\|=1$ and $\|v_l\|=l$ respectively. In particular,  
	\begin{equation}\label{truque}
		l ^\frac{N}{N-2s} \mathcal{S} _1 = \mathcal{S} _l.
	\end{equation}
\end{remark}
\subsection{Proof of Theorem~\ref{primeira}}
\begin{proof}
Suppose that $F(t)$ is self-similar and satisfies \eqref{posi_algum_auto}. Let $(u_k)\subset{\mathcal{D}^{s,2} (\mathbb{R}^N) }$ be a maximizing sequence for \eqref{max} with $l=1$, that is, $\|u_k\| ^2 = 1$ and $\Phi(u_k) \rightarrow \mathcal{S} _1.$ Let $(w ^{(n)}) _{n \in \mathbb{N}_{\ast } }$ in $\mathcal{D}^{s,2}(\mathbb{R}^N),$ $(y_k ^{(n)}) _ {k \in \mathbb{N}}$ in $\mathbb{Z}^N,$ and $(j _k ^{(n)}) _{k \in \mathbb{N}}$ in $\mathbb{Z},$ $n \in \mathbb{N}_{\ast },$ be the sequences provided by Theorem \ref{teo_tinta_frac}. By the Corollary \ref{rightone},
\begin{equation}\label{use1lemma1}
\mathcal{S} _1 = \lim _{k \rightarrow \infty} \Phi (u_k) = \sum _{n \in\mathbb{N}_{\ast }} \Phi (w^{(n)}),
\end{equation}
and at the same time by assertion \eqref{seis.tres},
\begin{equation}\label{use1seis.tres}
\sum _{n \in \mathbb{N}_{\ast}}\| w^{(n)} \|^2 \leq \limsup_k \|u_k \| ^2 \leq 1.
\end{equation}
The identity \eqref{use1lemma1} also implies that there exists $n \in \mathbb{N} _\ast$ with $w^{(n)}\neq 0.$ We may write $v ^{(n)}(x) = w ^{(n)} (\tau _n x )$ where $\tau _n = \|w^{(n)}\| ^{2/(N-2s)}.$ Consequently $\|v^{(n)}\|  ^2 =1,$ $\Phi (v^{(n)}) \leq \mathcal{S} _1$ and
\begin{equation*}\label{reuse1lemma1}
\mathcal{S} _1 = \sum _{n \in\mathbb{N}_{\ast }} \tau _n ^N \Phi (v^{(n)}) \leq \mathcal{S} _1 \sum _{n \in\mathbb{N}_{\ast }} \tau _n ^N.
\end{equation*}
Moreover,
\begin{equation}\label{sum1}
1 \leq \sum _{n \in\mathbb{N}_{\ast }} \tau _n ^N.
\end{equation}
From \eqref{use1seis.tres} we have
\begin{equation}\label{sum2}
\sum _{n \in\mathbb{N}_{\ast }} \tau _n ^{N-2s}\leq 1.
\end{equation}
Relations \eqref{sum1} and \eqref{sum2} can hold simultaneous provided that there is a $n_0 \in \mathbb{N_{\ast }}$ such that $\tau _{n_0} =1,$ while $\tau _n=0,$ whenever $n \neq n_0.$ Therefore, by \eqref{seis.quatro} we obtain 
\begin{equation*}
u_k - \gamma ^{\frac{N-2s}{2} j_k ^{(n_0)} } w^{(n_0)}(\gamma^{j _k ^{(n_0)}} ( \cdot - y_k^{(n_0)} ) ) \rightarrow 0\text{ in }L^{2^{\ast } _s} (\mathbb{R}^N).
\end{equation*}
Since $F(t)$ is self-similar, the sequence 
\begin{equation*}
v_k=\gamma ^{-\frac{N-2s}{2}j_k ^{(n_0)}}u_k (\gamma ^{- j_k ^{(n_0)}}x  + y_k^{(n_0)} ),
\end{equation*}
is a maximizing sequence for \eqref{max} and $v_k \rightarrow w^{(n_0)}$ in $L^{2^{\ast } _s} (\mathbb{R}^N).$ Furthermore, the continuity of $\Phi$ in $L^{2^{\ast } _s} (\mathbb{R}^N)$ implies $\Phi (w^{n_0}) = \mathcal{S} _1,$ and since $\| w^{(n_0)} \|^2  =1,\ w^{(n_0)}$ is a maximizer.\\
Consider now the case where $\mathcal{S} _1> \max \{\mathcal{S}_{1,+},\mathcal{S}_{1,-}\}.$ Let again $(u_k)$ be a maximizing sequence for \eqref{max} with $l=1.$ Since $f(t)$ verify \eqref{f_5} we can apply to Proposition \ref{lemma2} to get
\begin{equation}
\mathcal{S}_1 = \lim _{k \rightarrow \infty } \Phi (u _k) = \sum_{n \in \mathbb{N} _0} \Phi (w^{(n)}) + \sum_{n \in \mathbb{N} _{- \infty}} \Phi _{-} (w^{(n)}) + \sum_{n \in \mathbb{N} _{+ \infty}} \Phi _{+} (w^{(n)}),
\end{equation}
where $(w ^{(n)}),$ $(y_k ^{(n)}),$ $(j _k ^{(n)}),$ $n \in \mathbb{N}_{\ast },$ are given by Theorem \ref{teo_tinta_frac}. Considering again $v^{(n)}(x) = w^{(n)}(\tau _n x),$ with $\tau _n = \|w^{(n)}\| ^{2/(N-2s)},$ we get
\begin{equation}\label{kappapm}
1 \leq \sum_{n \in \mathbb{N} _0} \tau _n ^N + \frac{\mathcal{S}_{1,-}}{\mathcal{S} _1} \sum_{n \in \mathbb{N} _{- \infty}} \tau _n ^N + \frac{\mathcal{S}_{1,+}}{\mathcal{S} _1}\sum_{n \in \mathbb{N} _{+ \infty}} \tau _n ^N.
\end{equation}
Since $\mathcal{S} _{1, +} / \mathcal{S} _1 < 1$ and $\mathcal{S} _{1, -} / \mathcal{S} _1 < 1$ by assertion \eqref{seis.tres}, inequalities \eqref{use1seis.tres} and \eqref{kappapm} can hold simultaneously if and only if there is a $n _0 \in \mathbb{N} _0$ such that $\tau _{n _0} =1,$ while $\tau _n =0,$ whenever $n \neq n _0.$ Therefore, by assertion \eqref{seis.quatro} $u_k -w^{(n _0)} (\cdot - y _k ^{(n_0)}) \rightarrow 0$ in $L ^{2^{\ast} _s} (\mathbb{R}^N)$ and using a similar argument as in the previous case, we conclude that $w^{(n_0)}$ is a maximizer.
\end{proof}
\begin{remark}\label{2rmax+-}
	One always has $\mathcal{S} _1 \geq \max \{\mathcal{S}_{1,+},\mathcal{S}_{1,-}\}.$ Indeed, as discussed above, it suffices to prove this in the case that $l=1,$ so let $u \in \mathcal{D}^{s,2}{(\mathbb{R}^N)}$ with $\|u\|=1$ and $v_j := \delta _j u,$ where $\delta_j$ is given in \eqref{dilat_group}, and $j \in \mathbb{Z}$. Then $\|v_j\|=1$ implies that  $\Phi (v_j) \leq \mathcal{S} _1,$ and by condition \eqref{f_5} we conclude $\Phi (v_j) \rightarrow \Phi _{+} (u)$ as $j \rightarrow +\infty.$ The case for the inequality $\mathcal{S}_1\geq \mathcal{S}_{1,-}$ follows by using the same argument. Moreover, the inequality $\mathcal{S} _1 > \max \{\mathcal{S}_{1,+},\mathcal{S}_{1,-}\}$ holds whenever $F\geq F_{+}$ and $F\geq F_{-}$ with the strict inequality in a neighborhood of zero. In fact, since $F_+$ and $F_{-}$ are self-similar, we may consider $w _+ $ and $w_{-}$ the maximizers of $\mathcal{S} _{l, +}$ and $\mathcal{S} _{l, -},$ respectively, to obtain, by Theorem \ref{segunda}, Proposition \ref{prop_pohozaev} and Remark \ref{1rmax+-}, that $\mathcal{S} _{l, +} < \Phi (w_+) \leq \mathcal{S}_l$ and $\mathcal{S} _{l, -} < \Phi (w_{-}) \leq \mathcal{S}_l.$
\end{remark}
\subsection{Characterization of the minimax level}
We pass now to the study of the minimax level of the Lagrangian associated with Eq.~$\eqref{P},$ proving some useful results. This is made by considering the class of paths
$\zeta: [0,+\infty) \rightarrow \mathcal{D}^{s,2}(\mathbb{R}^N) $ defined by $\zeta_u(t)(x)=u( x /t)$
for any $u \in \mathcal{D}^{s,2}(\mathbb{R}^N),$ because of its homogeneous property with respect to the norm in $\mathcal{D}^{s,2}(\mathbb{R}^N).$
\begin{lemma}\label{path}
Suppose that $F(t)$ satisfies the growing condition \eqref{CRITG}. If $u \in \mathcal{D}^{s,2} (\mathbb{R}^N)$ is such that $\Phi(u)>0,$ then the path $\zeta _u$ belongs to $\Gamma _I.$ Thus $\Gamma _I \neq \emptyset$ if and only if \eqref{posi_algum_auto} holds.
\end{lemma}
\begin{proof}
Let $t_n, t_0 >0,$ $n \in \mathbb{N},$ be such that $t_n \rightarrow t_0$ and $u \in \mathscr{S}_0.$ Since
\begin{equation}\label{homog}
\|\zeta _u (t)\|^2 = t^{N-2s} \|u\|^2, \; \forall t>0,
\end{equation}
using \eqref{frac_id} we have
\begin{multline}\label{leminha}
\|\zeta _u (t_n) - \zeta_u (t_0)\|^2 = \\ t_n ^{N-2s} \|u\|^2  -2t_n^{-s} t_0 ^{-s}\int _{\mathbb{R}^N} (-\Delta )^{s/2} u ( x/t_n ) (-\Delta )^{s/2} u ( x/t_0 ) \dx  + t_0 ^{N-2s} \|u\|^2.
\end{multline}
Also, up to a set of Lebesgue measure zero, by identity \eqref{formula} we obtain
\begin{align*}
\left| (-\Delta )^{s} u (x/t_0 ) u ( x / t_n ) \right| &\leq \|u\| _{L^{\infty} (\mathbb{R}^N)} \left|(-\Delta )^{s} u ( x/t_0 ) \right|, \\
\lim _{n \rightarrow \infty} (-\Delta )^{s} u ( x/t_0 ) u ( x/t_n ) &= (-\Delta )^{s} u ( x/t_0 ) u ( x/t_0 ),\; \forall  x\in \mathbb{R}^N.
\end{align*}
Thus by the Dominated Convergence Theorem the left-hand side of the identity \eqref{leminha} goes to zero as $n \rightarrow \infty.$ By identity \eqref{homog} we conclude $\zeta _u \in C([0, \infty ),  \mathcal{D}^{s,2} (\mathbb{R}^N) ).$ The general case follows by a density argument.\\
Now suppose that \eqref{posi_algum_auto} holds. Then there exists $u \in C_0 ^{\infty} (\mathbb{R}^N)$ such that $\Phi (u) > 0 $ and consequently $\zeta _u \in \Gamma _I,$ since
\begin{align*}
I(\zeta _u (t)) = \frac{1}{2}t^{N-2s}\|u\|^2 - t^N \Phi(u) \rightarrow - \infty \text{ as } t \rightarrow \infty.
\end{align*}
Conversely, assume that $\Gamma _I \neq \emptyset.$ If \eqref{posi_algum_auto} does not hold, then we would have that $I(u) \geq 0,$ for all $u \in \mathcal{D}^{s,2} (\mathbb{R}^N). $ Hence $\Gamma _I = \emptyset,$ which is impossible.
\end{proof}
\begin{remark}
Let $u \in \mathcal{D}^{s,2} (\mathbb{R}^N)$ be such that $\Phi (u) >0.$ Then
\begin{equation}\label{zeros}
\max _{t \geq 0}I(\zeta _u (t)) = \frac{1}{2}\left(\frac{\|u\|^2}{2^\ast _s \Phi (u)} \right) ^{\frac{N-2s}{2s}} \| u \| ^2 - \left( \frac{\|u \| ^2}{ 2 ^\ast _s \Phi (u) } \right) ^{N/2s} \Phi (u).
\end{equation}
\end{remark}
\begin{lemma}\label{novolemma}
Assume that conditions \eqref{posi_algum_auto} and \eqref{CRITG} holds. Consider
\begin{equation*}
 \tilde{c}(I) := \inf _{\zeta \in \tilde{\Gamma}_I} \sup_{t \geq 0} I(\zeta (t)).
\end{equation*}
where
\begin{equation*}
\tilde{\Gamma}_I := \{\zeta \in \Gamma_I : \zeta = \zeta_u\ \text{ for some }u \in \mathcal{D}^{s,2}(\mathbb{R}^N) \text{ with } \Phi(u)>0 \}
\end{equation*}
Then $c(I)=\tilde{c}(I).$
\end{lemma}
\begin{proof}
	Since $\tilde{\Gamma}_I \subset \Gamma_I$ we have $c(I) \leq \tilde{c}(I).$ Suppose the contrary, that $c(I) < \tilde{c}(I).$ Then, there exists $\zeta \in \Gamma_I$ such that $c(I) \leq \sup _{t \geq 0} I(\zeta (t)) < \tilde{c}(I).$ 	Observe now that, by \eqref{embds} and \eqref{f_2}, the continuous function
	\begin{equation*}
	h(t) = \frac{1}{2} \| \zeta (t) \| ^2 - \frac{2 ^\ast _s}{2} \Phi (\zeta (t)), \ t>0,
	\end{equation*}
	changes sign. Hence, there exists $t_0>0$ such that $g(t_0) = 0$ and $\zeta(t_0) \neq 0,$ which implies that $\| \zeta (t_0) \| ^2 = 2^\ast _s \Phi (\zeta (t_0)).$ Now taking $u= \zeta (t_0)$ in \eqref{zeros} we get
	\begin{equation*}
	\sup _{t\geq 0} I(\zeta _u (t) ) =  \frac{1}{2} \| \zeta (t_0) \| ^2 - \Phi (\zeta (t_0)) \leq \sup _{t \geq 0} I(\zeta (t)),
	\end{equation*}
	which leads to a contradiction with the definition of $\tilde{c} (I).$
\end{proof}
\begin{remark}\label{remark_minimax} In order to prove our nonlocal counterpart of 
\cite[Proposition 2.4]{tintapaper}, we have to reduce the class of admissible paths. This is made by noticing that
\begin{equation*}
\sup_{t \geq 0}I(\zeta_v (t)) = \sup_{t \geq 0}I(\zeta_{v_\sigma} (t)),
\end{equation*}
for any rescaling $v_\sigma (x) = v(x/\sigma),\ \sigma > 0,$ and taking account the set
\begin{equation*}
\tilde{\Gamma}^1_I := \left\lbrace \zeta \in \Gamma_I : \zeta = \zeta_u\ \text{ for some }u \in \mathcal{D}^{s,2}(\mathbb{R}^N) \text{ with } \Phi(u)>0\text{ and }\|u\|\geq 1 \right\rbrace ,
\end{equation*}
and the associated minimax level
\begin{equation*}
\tilde{c}_1 (I) := \inf _{\zeta \in \tilde{\Gamma}^1_I} \sup_{t \geq 0} I(\zeta (t)),
\end{equation*}
to obtain that $ \tilde{c} (I)= \tilde{c}_1 (I).$
\end{remark}
\subsection{Proof of Theorem~\ref{segunda}}
\begin{proof}
\ref{item1a} Let $ v \in  \mathcal{D}^{s,2} (\mathbb{R}^N) $ be a non-trivial critical point of $I.$ By Proposition \ref{prop_pohozaev} we have $\zeta _v \in \Gamma_I$ and $t=1$ is a maximum point for the function $t \mapsto I (\zeta _v (t)) = t ^{N-2s} \| v \|^2 /2- t^N \Phi (v).$ Hence $c(I) \leq \max _{t \geq 0} I( \zeta _v (t)) = I(v).$\\
\ref{item2a} Since $w$ is a maximizer for \eqref{max} we have
\begin{equation}\label{lagrange}
\int _{\mathbb{R}^N} f(w) v \dx=2\lambda \int _{\mathbb{R}^N} (-\Delta )^{s/2} w (-\Delta )^{s/2} v \dx, \; \forall  v \in \mathcal{D}^{s,2} (\mathbb{R}^N),
\end{equation}
where $\lambda$ is a Lagrange multiplier. We claim that $\lambda \neq 0.$ Indeed, on the contrary, we get $f(w)=0$ a.e in $\mathbb{R}^N,$ which leads to a contradiction with $\Phi(w)>0.$ Thus, we can apply Proposition \ref{prop_pohozaev} to get
\begin{equation*}
2\lambda \|w\| ^2 =2^\ast _s \int _{\mathbb{R}^N} F(w) \dx,
\end{equation*}
which together with relation \eqref{truque} implies $2\lambda l_0 = 2 ^\ast_s S _1 l_0 ^{N/(N-2s)},$ and the explicit value of $l_0$ gives $\lambda=1.$ In particular,
\begin{equation*}
I(w)=\left(\frac{1}{2} - \frac{1}{2_s ^\ast} \right)\| w \| ^2>0.
\end{equation*}
Let us prove now the last statement of \ref{item2a}. By the part \ref{item1a}, it is sufficient to prove that $I(w) \leq c(I).$ Let $u \in \mathcal{D}^{s,2}(\mathbb{R}^N)$ with $\Phi(u) >0,$ and denote $\tilde{u}(x) = u (\alpha x)$ where $\alpha = \|u\| ^{2/(N-2s)}.$ Then $\| \tilde{u} \| = 1$ and consequently
\begin{equation*}
\Phi (\zeta _u (t)) = \Phi (\zeta _{\tilde{u} }(t \alpha ) ) \leq \|\zeta _u (t) \|^{\frac{2N}{N-2s}} \mathcal{S}_1, \; \forall t\geq 0,
\end{equation*}
from which we can deduce, by Lemma \ref{novolemma} and Remark \ref{remark_minimax}, that
\begin{align*} 
{c} (I) &= \inf _{\Phi (u) >0,\atop{  \|u\|\geq 1 }} \sup _{t \geq 0} \frac{1}{2} \|\zeta _u (t)\|^2 - \Phi (\zeta _u (t))\\
&\geq \inf _{\Phi (u) >0,\atop{  \|u\|\geq 1 }} \sup _{t \geq 0}\frac{1}{2}\|\zeta _u (t)\|^2 - \|\zeta _u (t) \| ^{\frac{2N}{N-2s}} \mathcal{S} _1.
\end{align*}
Moreover, we have
\begin{multline*}
\sup _{t \geq 0}\frac{1}{2}\|\zeta _u (t)\|^2 - \|\zeta _u (t) \| ^{\frac{2N}{N-2s}} \mathcal{S} _1 \\ = \left[\frac{1}{2} (2_s ^\ast \mathcal{S}_1 ) ^{-\frac{N-2s}{2s} } - \mathcal{S}_1 (2^\ast _s \mathcal{S} _1)^{-\frac{N}{2s} }   \right]\|u\|^{2_s ^\ast (1-s)},\; \forall
 u\in D^{s,2} (\mathbb{R}^N)\text{ with }\Phi(u)>0.
\end{multline*}
Consequently,
\begin{equation*}
\inf _{\Phi (u) >0,\atop{  \|u\|\geq 1 }} \sup _{t \geq 0}\frac{1}{2}\|\zeta _u (t)\|^2 - \|\zeta _u (t) \| ^{\frac{2N}{N-2s}} \mathcal{S} _1 = \frac{1}{2} (2_s ^\ast \mathcal{S}_1 ) ^{-\frac{N-2s}{2s} } - \mathcal{S}_1 (2^\ast _s \mathcal{S} _1)^{-\frac{N}{2s} }.
\end{equation*}
On the other hand, by the explicit value of $l_0$ and relation \eqref{truque} we have that
\begin{equation*}
I(w) = \frac{1}{2} (2_s ^\ast \mathcal{S}_1 ) ^{-\frac{N-2s}{2s} } - \mathcal{S}_1 (2^\ast _s \mathcal{S} _1)^{-1\frac{N}{2s} }.
\end{equation*}
Thus $c(I) = I(w)$ and by the proof of the statement \ref{item1a}, the path $\zeta _w \in \Gamma_I$ is minimal.
\end{proof}
\subsection{Proof of Theorem~\ref{minimizacao}}
\begin{proof}
We start by noting that the embedding \eqref{embds} together with condition \eqref{f_2} implies $\mathcal{I}>0.$ Let $(u_k)$ be a minimizing sequence, that is, $\Phi(u_k ) =1$ and $\|u_k \|^2 \rightarrow \mathcal{I}.$ Since this sequence is bounded, we may apply Theorem \ref{teo_tinta_frac} to obtain the weak profile described in \eqref{seis.um}--\eqref{seis.quatro}. By the Corollary \ref{rightone}, we have
\begin{equation*}
1= \sum _{n \in \mathbb{N} _{\ast }} \int _{\mathbb{R}^N} F(w^{(n)}) \dx,
\end{equation*}
which implies that there exists $n \in \mathbb{N} _\ast$ with $0<\Phi (w^{(n)})\leq 1.$ If $\Phi (w^{(n)}) =1,$ considering $d_k$ as the element of $D_{\mathbb{Z}^N, \mathbb{R}}$ given by assertion \eqref{seis.um}, we have by the weak lower semi-continuity of the norm that
\begin{equation*}
\mathcal{I} \leq \|w^{(n)} \| ^2 \leq \liminf _k \|d^\ast _k u_k \|^2 = \mathcal{I}\quad\text{and}\quad\|d^\ast _ku_k \| ^2 = \| u_k \|^2 \rightarrow \|w^{(n)} \| ^2,
\end{equation*}
which proves the first part of Theorem \ref{minimizacao}. Hence, let us assume that $\Phi (w^{(n)})< 1.$ Set $v_k = d_k ^\ast u_k -w,$ where $w=w^{(n)}.$ By Corollary \ref{brezis-lieb} we have
\begin{equation}\label{Klarge}
\lim_{k \rightarrow \infty}\left[   1 - \int _{\mathbb{R}^N} F(v_k) \dx \right] = \int _{\mathbb{R}^N} F(w) \dx 
\end{equation}
Denote $\delta =\Phi(w)$ and set $\hat{w}(x)=w(\delta ^{1/N} x).$ Thus $\Phi (\hat{w}) =1$ and consequently
\begin{equation}\label{min1}
\|w \|^2= \delta ^{\frac{N-2s}{N}} \| \hat{w}\|^2 \geq \delta ^{\frac{N-2s}{N}} \mathcal{I}.
\end{equation}
Now consider
\begin{equation*}
\hat {v} _k (x)= v_k (|1- \delta |^{1/N} \beta _k ^{1/N} x ),\text{ where }\beta _k = \Phi (b_k)\text{ and }b_k(x)=v_k (|1-\delta|^{1/N} x ).
\end{equation*}
Since $\beta _k = |1-\delta |^{-1}\Phi(v_k),$ by convergence \eqref{Klarge} we have $\beta _k \rightarrow 1,$ and we conclude $\Phi (\hat {v} _k) =1$ for large $k.$ This leads to
\begin{equation}\label{min2}
\| v_k \|^2 = |1-\delta | ^{\frac{N-2s}{N}} \beta _k ^{\frac{N-2s}{N}} \| \hat {v} _k  \| ^2 \geq |1-\delta | ^{\frac{N-2s}{N}} \beta _k ^{\frac{N-2s}{N}} \mathcal{I},
\end{equation}
for large $k.$ In the other hand, since $\| u_k \| ^2 = \|d_k ^\ast u_k \|^2,$ by relations \eqref{min1} and \eqref{min2} we may infer
\begin{align*}
\| u_k \| ^2 &= \| v_k \| ^2 +2(v_k, w) + \|w\|^2\\
&\geq \left(\delta ^{\frac{N-2s}{N}}+ |1-\delta | ^{\frac{N-2s}{N}}\beta _k ^{\frac{N-2s}{N}} \right) \mathcal{I},
\end{align*}
and passing to the limit we finally conclude
\begin{equation*}
1\geq \delta ^{1-\frac{2s}{N}} + |1-\delta | ^{1-\frac{2s}{N}},
\end{equation*}
which leads to a contradiction since $0<\delta <1.$ Thus $w$ is the minimizer in \eqref{min} and consequently we have
\begin{equation*}
\int _{\mathbb{R}^N} (-\Delta )^{s/2} w (-\Delta )^{s/2} v \dx = \lambda \int _{\mathbb{R}^N} f(w) v \dx, \; \forall v \in \mathcal{D}^{s,2} (\mathbb{R}^N),
\end{equation*}
where $\lambda \in \mathbb{R}$ is a Lagrange multiplier. Taking $v=w$ in the above identity we have $\lambda \neq 0,$ which allows us to apply Proposition \ref{prop_pohozaev} to get $\lambda = \mathcal{I}/2^\ast _s,$ which by an easy computation using identities \eqref{frac_id} leads us to conclude that $u(x) = w( x /\beta)$ is a non-trivial weak solution of Eq.~\eqref{P}, where $\beta = \lambda ^{1/2s}= (\mathcal{I}/ 2^\ast _s ) ^{1/2s}.$ \\Let us prove now that $u(x)=w(x / \beta )$ is a ground state solution of Eq.~\eqref{P}. We start by applying Proposition \ref{prop_pohozaev} again to obtain
\begin{equation}\label{gs1}
I(u) = \left(\frac{1}{2} -\frac{1}{2_s ^\ast}\right) \| u \| ^2 =\frac{s}{N} (2 _s ^\ast) ^{-\frac{N-2s}{2s}}\| w \| ^{N/s}.
\end{equation}
Now let $v \in \mathcal{D}^{s,2} (\mathbb{R}^N)$ be any non-trivial weak solution of Eq.~\eqref{P}. For any $\sigma>0$ denote $v _\sigma(x) = v(x / \sigma).$ Choose $\sigma$ such that $\Phi( v _\sigma) = 1,$ that is, $\sigma =  (\Phi (v))^{-1/N}.$ Replacing this value of $\Phi(v)$ in the identity $\|v \|^2 = 2 ^\ast _s \Phi (v),$ we get $\sigma  = (2^\ast _s )^{1/N} \| v \| ^{-2/N}.$ Consequently, we obtain
\begin{equation*}
\| v _\sigma \| ^2 = (2_s ^\ast ) ^{\frac{N-2s}{N}} (\|v\| ^2) ^{2s/N},
\end{equation*}
which implies
\begin{equation}\label{gs2}
I(v)=\frac{s}{N} \| v \| ^2 =\frac{s}{N} (2 _s ^\ast) ^{-\frac{N-2s}{2s}} \| v _\sigma \| ^{N/s}.
\end{equation}
Comparing identities \eqref{gs1} and \eqref{gs2}, we conclude that $I(u) \leq I(v),$ i.e, $u$ is a ground state solution for Eq.~\eqref{P}. 
\end{proof}
\section{The non-autonomous case}\label{tnac}
For the sake of discussion, we are going to compare the minimax level of the asymptotic functional $I_\kappa,$ with the minimax of the Lagrangian associated with Eq.~\eqref{P}, for $\kappa=0,+,-.$
\begin{proposition}\label{minimaxprop}
Suppose that \eqref{f_1}--\eqref{f_pohozaev} holds. If $F_0$ is self-similar or $(F_0)_\kappa(t) \leq F_\kappa (t),$ for all $t,$ $\kappa = +,-,$ then $c(I) \leq c(I_\kappa),$ for $\kappa= 0,+,-.$ Moreover, under these assumptions, \eqref{f'_6} implies \eqref{suficient}.
\end{proposition}
\begin{proof}
Let be $\mathcal{S} _l^\kappa,$ the associated constrained maximum similar to \eqref{max} relative to the primitive $F_\kappa,$ precisely,
\begin{equation*}\label{maxc}
\mathcal{S}_l^\kappa = \sup _{\|u\|^2 =l} \int _{\mathbb{R}^N}F_\kappa (u ) \dx\quad \mbox{for}\quad \kappa=0,+,-.
\end{equation*}
For each $\kappa=+,-,$ the primitive of the nonlinearity  $F_\kappa$ is auto-similar, thus using Theorems \ref{primeira} and \ref{segunda}, we conclude that there exists $w_\kappa$ maximizer of $\mathcal{S} _{{l_0} }^\kappa $ such that 
\begin{equation*}
c(I _\kappa) = I_\kappa( w_ \kappa ) = \max _{t \geq 0} I_\kappa( \zeta _{w _\kappa } (t))>0.
\end{equation*}
	For each $\kappa=+,-,$ let us consider the sequence
	\begin{equation*}
	w ^{\kappa }_n (x) := \gamma ^{\frac{N-2s}{2} j ^\kappa_n} w _\kappa \left( \gamma ^{j^\kappa_n} x  \right) ,
	\end{equation*}
	where the sequence $(j^\kappa_n) \subset \mathbb{Z}$ is chosen in such a way that $j_n ^+ \rightarrow + \infty$ and  $j_n ^- \rightarrow - \infty$.
 Since for each $\kappa=+,-,$ 
	\begin{equation}\label{compac_path}
	\left|I( \zeta _{w^\kappa _n} (t)) -  I_\kappa( \zeta _{w _\kappa } (t) ) \right| \leq t^N \int _{\mathbb{R}^N }  \left| \gamma ^{-N j^\kappa_n} F \left(  \gamma ^{-j^\kappa_n} tx ,\gamma ^{\frac{N-2s}{2} j^\kappa_n} w_\kappa (x) \right) -   F_\kappa (w _\kappa(x)) \right| \dx,
	\end{equation}
	the uniformity assumption on the limits in \eqref{f_5}, guarantees (by a density argument) that 
	\begin{equation}\label{conv_compat}
	\lim_{n\rightarrow \infty}I( \zeta _{w^\kappa _n} (t)) = I_\kappa( \zeta _{w _\kappa } (t)),\quad \text{uniformly in compact sets of }\mathbb{R}.
	\end{equation}
We also have that the path $\zeta _{w^\kappa _n},$ $\kappa=+,-,$  belongs to $\Gamma_I,$ for $n$ large enough. In fact, by the uniformly convergence in $x$ of \eqref{f_5} and Proposition \ref{prop_pohozaev}, there exists $n_0>0$ such that
\begin{equation*}
\int _{\mathbb{R}^N }  \gamma ^{-N j^\kappa_n} F \left(  \gamma ^{-j^\kappa_n} tx ,\gamma ^{\frac{N-2s}{2} j^\kappa_n} w_\kappa (x) \right) \dx > \frac{1}{2}\int _{\mathbb{R}^N } F_\kappa(w_\kappa)\dx,\text{ for all }n>n_0\text{ and }t>0.
\end{equation*}
Thus, for each $n$ there exist $t_n>0$ such that
	\begin{equation*}
	I(\zeta _{w^\kappa _n} (t_n) )= \max_{t \geq 0} I(\zeta _{w^\kappa _n} (t))>0.
	\end{equation*}
We claim that the sequence $(t_n)$ is bounded. On the contrary, up to a subsequence, we get the following contradiction
\begin{equation*}
0<I( \zeta _{w^\kappa _n} (t_n))= \frac{1}{2} t_n^{N-2s} [w]_s ^2 -t_n^N \int_{\mathbb{R}^N}\gamma ^{-N j_n} F \left(  \gamma ^{-j_n} t_nx ,\gamma ^{\frac{N-2s}{2} j_n} w_\kappa (x) \right)\dx\\ \rightarrow - \infty,\text{ as }n\rightarrow \infty.
\end{equation*}
Therefore, up to a subsequence, $t_n\rightarrow t_0,$ and we have	
	\begin{equation*}
	\lim_{n\rightarrow \infty}\max _{t \geq 0} I(\zeta _{w^\kappa _n} (t_n))=I_{\kappa} (\zeta _{w _\kappa} (t_0)),
	\end{equation*}
	because of \eqref{conv_compat}. Thus we may conclude
	\begin{equation*}
	c(I) \leq \lim_{n\rightarrow \infty}\max _{t \geq 0} I(\zeta _{w^\kappa _n} (t)) \leq \max _{t \geq 0} I_\kappa( \zeta _{w _\kappa } (t)) = c(I _\kappa).
	\end{equation*}
	If there exists maximizer $w_0$ for $\mathcal{S} _{l_0}^0,$ then an similar argument as above leads to $c(I) \leq c(I_0).$ In fact, for each $n,$ define the path
	\begin{equation*}
	\lambda _n (t) = w_0 \left(\frac{\cdot - y_n}{t}\right) ,\quad t \geq 0,
	\end{equation*}
	where $(y_n)$ is taken in a such way that $|y_n|\rightarrow\infty.$ As before, we consider the estimate
	\begin{equation*}
	\left| I( \lambda _n (t) ) - I_0 (w_0 (\cdot / t )) \right| \leq t^N\int _{\mathbb{R}^N } \left| F(tx + y_n,w_0) - F_0 (w_0) \right| \dx,
	\end{equation*}
to obtain that
	\begin{equation*}
	\lim _{n \rightarrow \infty } I( \lambda _n (t) ) = I_0 (w_0(\cdot / t )),\quad \text{uniformly in compact sets of }\mathbb{R}.
	\end{equation*}
	We also have that the path $\lambda _n$ belongs to $\Gamma_I,$ for $n$ large enough. Indeed, assuming the contrary, we would obtain $n_0$ and a sequence $l_n \rightarrow \infty$ such that $I(\lambda_{n_0} (l_n)) > 0,$ for all $n.$ On the other hand, we have that
	\begin{equation*}
	\lim_{n \rightarrow \infty} \int_{\mathbb{R}^N}F(l_nx+y_{n_0},w_0)\dx = \int_{\mathbb{R}^N}F_0(w_0)  \dx,
	\end{equation*}
	which, by taking $n$ large enough, leads to the contradiction $I(\lambda_{n_0} (l_n)) < 0$. Let $t_n>0$ such that 
	\begin{equation*}
	I(\lambda _n (t_n) )= \max_{t \geq 0} I(\lambda _n (t))>0.
	\end{equation*}
	Once again we get that the sequence $(t_n)$ is bounded. On the contrary, there is a subsequence $(t_{k_n})$ that implies in the following contradiction
	\begin{equation*}
	0<I(\lambda_n (t_{k_n}))= \frac{1}{2} t_{k_n}^{N-2s} [w_0]_s ^2 -t_{k_n}^N \int_{\mathbb{R}^N}F(t_{k_n}x+y_n,w_0) \dx \rightarrow - \infty,\text{ as }n\rightarrow \infty.
	\end{equation*}
	Thus, up to a subsequence, $t_n\rightarrow t_0$ and we obtain that
	\begin{equation*}
	\lim_{n\rightarrow \infty}\max _{t \geq 0} I(\lambda _n (t))=I _0 (w_0(\cdot/t_0)).
	\end{equation*}
	As a consequence we conclude that
	\begin{equation*}
	c(I) \leq \lim_{n\rightarrow \infty}\max _{t \geq 0} I(\lambda _n (t_n)) \leq \max _{t \geq 0} I_0 (w_0(\cdot / t)) = c(I_0),
	\end{equation*}
	where we used Proposition \ref{prop_pohozaev} to induce that $t=1$ is the unique critical point of $ I_0 (w_0(\cdot / t)).$ Thus, let us assume that $\mathcal{S} _{l_0}^0$ is not attained. By the Remarks \ref{1rmax+-} and \ref{2rmax+-}; and Theorem \ref{primeira}, if $\mathcal{S} _{l_0}^0 $ is not attained then $\mathcal{S} _{l_0} ^0 = \mathcal{S} _{l_0} ^+$ or $\mathcal{S}_{l_0} ^0 = \mathcal{S}_{l_0} ^{-}.$ Thus, using the definition of $\mathcal{S} _{l_0}^0$ we get
	\begin{equation*}
	c(I _\kappa) \leq I_0 (u),\ \kappa=+,-, \; \forall u \in \mathcal{D}^{s,2} (\mathbb{R}^N)\text{ with } \|u\|^2 = l_0.
	\end{equation*}
Let $u \in \mathcal{D}^{s,2} (\mathbb{R}^N),\ u\neq 0,$ and denote $\alpha=\|u\|^2,$ then considering the rescaling $u_{l_0}(x) = u(t_0 x),$ where $t_0 = (\alpha/l_0) ^{-1/(N-2s)},$ we have $\| u _{l_0}\|^2 = l_0$ and consequently 
\begin{align*}
c(I_{\kappa} ) \leq I(u_{l_0}) &=\frac{1}{2}t_0^{N-2s} \|u\| ^2 - t_0 ^N  \Phi _0(u)\\
&\leq \max _{t\geq 0} I_0( \zeta _u (t)),\text{ for }\kappa=+,-.
\end{align*}
By Lemma \ref{novolemma} we conclude $c(I_{\kappa} ) \leq c(I_0),\ \kappa=+,-.$\\
Now suppose that \eqref{f'_6} holds. As seen above, $\zeta_{w_\kappa}$ belongs to $\Gamma_I,$ thus
\begin{equation*}
c(I) \leq \max _{t \geq 0} I(\zeta _{w_\kappa} (t) ) <  \max _{t \geq 0} I _\kappa (\zeta_{w_\kappa} (t))=c(I_\kappa),\ \kappa=+,-.
\end{equation*}
We claim that $S_{l_0} ^0$ is attained, from which we conclude the desired inequality in \eqref{suficient}. Assume the contrary, by arguing as before, we have $\mathcal{S} _{l_0} ^0 = \mathcal{S} _{l_0} ^+$ or $\mathcal{S}_{l_0} ^0 = \mathcal{S}_{l_0} ^{-}.$ Taking $|x|\rightarrow  \infty$ in \eqref{f'_6} we get that $F_0 (t) \geq F_{\kappa} (t),\ \kappa=+,-,$ for all $t \in \mathbb{R}.$ Consequently, in any case,
\begin{align}\label{last_contra}
\int _{\mathbb{R}^N} F_0 (w_\kappa) \dx &\leq \sup _{\|u\|^2=l_0} \int _{\mathbb{R}^N} F _0 ( u ) \dx \nonumber\\
&=\int _{\mathbb{R}^N} F _\kappa (w _\kappa) \dx \leq \int _{\mathbb{R}^N} F _0 (w _\kappa) \dx,\ \kappa=+,-,
\end{align}
a contradiction, because relation \eqref{last_contra} implies that $S^0 _{l_0}$ is attained.\end{proof}
Summarizing all the discussion until now we can finally prove Theorem \ref{principal}.
\subsection{Proof of Theorem~\ref{principal}}
In order to treat the case without compactness condition \eqref{suficient}, that is not considered in the local counterpart \cite{tintapaper}, where the case $c(I_\kappa)=c(I),$ $\kappa=0,+,-,$ may occur, we need the following result, which states that the existence of a critical point of $I$ is guaranteed whenever the minimax level \eqref{minimax} is attained.
\begin{theoremletter}\cite[Theorem~2.3]{lins}\label{p_local_MP}
Let $E$ be a real Banach space. Suppose that $I \in C^1(E)$ satisfies
\begin{enumerate}[label=(\roman*)]
\item $I(0)=0;$
\item There exists $r,\ b>0$ such that $I(u) \geq b,$ whenever $\|u\| = r; $
\item There is $e\in E$ with $\|e\| > r$ and $I(e)<0;$ 
\end{enumerate}
Let
\begin{equation*}
c(I) = \inf_{\gamma \in \Gamma} \sup_{t\in [0,1]} I(\gamma (t)),
\end{equation*}
where
\begin{equation*}
\Gamma = \left\lbrace \gamma \in C([0,1], E) : \gamma (0) = 0,\ \|\gamma(1)\| > r,\ I(\gamma (1))<0 \right\rbrace .
\end{equation*}
If there exists $\gamma _0 \in \Gamma$ such that
\begin{equation*}
c = \max _{t \in [0,1]} I(\gamma _0 (t) ),
\end{equation*}
then $I$ possess a non-trivial critical point $u\in \gamma _0 ([0,1]) $ such that $I(u) = c.$
\end{theoremletter}
\begin{remark}
We define
\begin{equation*}
c_1(I)= \inf_{\gamma \in \Gamma_I } \sup_{t \in [0,1]} I(\gamma (t)),
\end{equation*}
where
\begin{equation*}
\Gamma^1 _{I} = \left\lbrace \gamma \in C([0,1],  \mathcal{D}^{s,2}(\mathbb{R} ^N) ) : \gamma (0) = 0,\ \|\gamma(1)\| > r,\ I(\gamma (1))<0 \right\rbrace,
\end{equation*}
as the usual minimax level. We have that $c_1(I)= c(I).$
\end{remark}
\begin{proof}[Proof of Theorem~\ref{principal} completed] For the readers convenience, we divide the proof in several steps. 

(i) Let us assume first that condition \eqref{suficient} holds true. We start observing that the assumptions \eqref{AR} and \eqref{posi_algum} implies that the functional $I$ has the mountain pass geometry. In particular, $\Gamma _I \neq \emptyset$ and $0<c(I)<\infty.$ In fact, set $v=\varphi_R(x) := v_R(|x-x_0|),$ where $v_R$ as defined as in Remark \ref{remark_positivo}. Then $\varphi_R \in \mathcal{D}^{s,2} (\mathbb{R}^N )$ and we have
\begin{align*}
\int_{\mathbb{R}^N} F(x,v ) \dx &= \int _{B_R (x_0)} F(x,t_0) \dx + \int_{B_{R+1}(x_0)\setminus B_R(x_0)} F(x,v) \dx\\
&\geq |B_R|\inf _{B_R(x_0)} F(x,t_0) + |B_{R+1}\setminus B_R| \inf_{(x,t) \in (B_{R+1}(x_0) \setminus B_R(x_0)) \times [0,t_0]} F(x,t) >0
\end{align*}
Since \eqref{AR} is equivalent to $d/dt (F(x,t)t ^{-\mu}) \geq 0,$ $t>0,$ we have for $t>1$ that
\begin{equation*}
\int_{\mathbb{R}^N}F(x,tv)\dx \geq t^\mu \int_{\mathbb{R}^N} F(x,v)\dx.
\end{equation*}
Hence
\begin{equation*}
I(tv)=\frac{t^2}{2}\|v\|^2 - \int_{\mathbb{R}^N} F(x,tv)\dx \leq \frac{t^2}{2}\|v\|^2 - t^\mu \int_{\mathbb{R}^N} F(x,v)\dx \rightarrow - \infty,\text{ as }t\rightarrow \infty.
\end{equation*}
In the other hand, by the growth condition \eqref{f_2} and the embedding \eqref{embds},
\begin{equation*}
I(u)\geq \|u\|^2 \left( \frac{1}{2} - C \|u\| ^{2_s ^{\ast }-2} \right),\quad u \in \mathcal{D}^{s,2}(\mathbb{R}^N),
\end{equation*}
for some constant $C>0.$ Thus, choosing $\|u\|$ sufficiently small, we have $I(u)>0.$ The same can be concluded for the functionals $I_\kappa,$ since $F_\kappa$ satisfies \eqref{AR} and \eqref{posi_algum}.\\
Let $(u_k)$ in $\mathcal{D}^{s,2} (\mathbb{R}^N)$ be such that $I(u_k) \rightarrow c(I)$ and $I'(u_k) \rightarrow 0,$ which the existence can be guaranteed by the Mountain Pass Theorem (see \cite{ambrorab}).

(ii) By assumption \eqref{AR}, this sequence is bounded in $\mathcal{D}^{s,2} (\mathbb{R}^N),$ since for large $k,$ we have
\begin{align*}
c(I)+1+\| u_ k \|  &\geq I(u_k) - \frac{1}{\mu} I'(u_k) \cdot u_k \\
&=\left(\frac{1}{2} - \frac{1}{\mu} \right) \|u_k\| ^2 - \int _{\mathbb{R}^N} F(x,u_k) - \frac{1}{\mu} f(x,u_k) u_k \dx\\
&\geq \left( \frac{1}{2} - \frac{1}{\mu } \right) \| u_k \| ^2.
\end{align*}
Let $(w ^{(n)}),$ $(y_k ^{(n)})$ and $(j _k ^{(n)})$ be the sequences provided by Theorem \ref{teo_tinta_frac}. If $w^{(n)}= 0$ for all $n\geq 2,$ then by assertion \eqref{seis.quatro} and \eqref{seis.um}, 
\begin{equation*}
u_k \rightarrow w^{(1)}\text{ in } L^{2_s ^\ast} (\mathbb{R}^N) \text{ and }u_k \rightharpoonup w^{(1)}\text{ in }\mathcal{D}^{s,2}(\mathbb{R}^N).
\end{equation*}
Therefore we conclude that $w^{(1)}$ is a critical point of $I$ such that, up to a subsequence, $u_k \rightarrow w^{(1)}$ in $\mathcal{D}^{s,2} (\mathbb{R}^N).$

(iii) Let us argue by contradiction and assume that there exists $n_0\geq 2,$ such that $w^{(n_0)} \neq 0.$ By the estimate \eqref{seis.tres} and Proposition \ref{lemma2} we have, up to a subsequence, that
\begin{align}
c(I)&=\lim _{k \rightarrow \infty}\left[  \frac{1}{2}\|u_k\|^2 - \int_{\mathbb{R}^N} F(x,u_k) \dx \right]  \nonumber \\
&\geq I(w^{(1)}) + \sum _{n \in \mathbb{N} _0,n >1} I_0 (w^{(n)}) + \sum _{n \in \mathbb{N} _{+}} I_+ (w^{(n)}) + \sum _{n \in \mathbb{N} _{-} } I _{-} (w^{(n)})\label{liminf}.
\end{align}
Let $\varphi \in C _0 ^\infty (\mathbb{R}^N)$ and $n \geq 1.$ Since
\begin{equation*}
\left| \gamma ^{-\frac{N+2s}{2}j_k ^{(n)}} f\left(\gamma ^{-j_k ^{(n)}}x + y_k ^{(n)} ,\gamma ^{\frac{N-2s}{2} j_k ^{(n)}} t  \right) \right| \leq C |t|^{2_s ^\ast -1}, \; \forall x \in \mathbb{R}^N\text{ and }t\in \mathbb{R},
\end{equation*}
by the embedding $\eqref{comp_loc},$ we can take the limit
\begin{align*}
& I' (u_k) \cdot \left(\gamma ^{\frac{N-2s}{2} j_k ^{(n)}} \varphi (\gamma^{j _k ^{(n)}} ( \cdot - y_k^{(n)} ) ) \right)  \\
&= \left(\gamma ^{-\frac{N-2s}{2}j_k ^{(n)}}u_k (\gamma ^{- j_k ^{(n)}} \cdot + y_k^{(n)} ) , \varphi \right)
- \int _{\mathbb{R}^N} \gamma ^{-\frac{N+2s}{2}j_k ^{(n)}} f\left(\gamma ^{-j_k ^{(n)}}x + y_k ^{(n)} ,\gamma ^{\frac{N-2s}{2} j_k ^{(n)}} v_k ^{(n)}  \right)\varphi \,\dx ,
\end{align*}
where
\begin{equation*}
v_k ^{(n)}(x):= \gamma ^{-\frac{N-2s}{2} j_k ^{(n)}} u_k (\gamma ^{-j_k ^{(n)}}x + y_k ^{(n)}  ),
\end{equation*}
to conclude that $w^{(1)}$ is a critical point of $I$ and $w^{(n)}$ is a critical point of $I_0, I_{+}$ or $I_{-},$ provided that $n \in \mathbb{N}_0, \mathbb{N}_+$ or $\mathbb{N}_{-},$ respectively. Consequently, using assumption \eqref{AR}
\begin{equation*}
I_\kappa  (w^{(n)}) = \frac{1}{2} \int _{\mathbb{R} ^N} f_\kappa(w^{(n)})w^{(n)} \dx - \int_{\mathbb{R}^N} F_\kappa (w^{(n)}) \dx \geq 0,\; \forall n\geq 2,
\end{equation*}
and $I(w^{(1)})\geq 0.$ On the other hand, the assumption $c(I) < c(I_\kappa)$ and the estimate \eqref{liminf} implies $I_ \kappa (w^{(n_0)}) < c(I_ \kappa),$ which leads to a contradiction with Theorem \ref{segunda}.

(iv) Suppose now that relation \eqref{desigualdade_geral} holds instead of \eqref{suficient}. Condition \eqref{desigualdade_geral} implies that the path $\zeta_{w^{(n_0)}}$ belongs to $\Gamma_I$ and $c(I) \leq I\kappa (w^{(n_0)}),$ where $\kappa$ is the corresponding index for which $n_0$ belongs. In view of the above discussion and estimate \eqref{liminf}, we conclude that
\begin{equation*}
u_k \rightarrow w^{(1)}\text{ in a subsequence}\quad \text{or}\quad c(I) = \max _{t \geq 0} I(w^{(n_0)} (\cdot / t)).
\end{equation*}
If the minimax level $c(I)$ is attained then we can apply Proposition \ref{p_local_MP} to obtain the existence of a critical point $u \in \zeta _{w^{(n_0)}}([0,\infty )) $ such that $I(u) = c(I).$
\end{proof}

%\bibliographystyle{ams_ex}
%\bibliography{references}
\end{document}